\documentclass[11pt]{article}
\usepackage{graphicx}
\usepackage{amsmath,amsthm,latexsym,amsopn,amssymb,amsxtra,amstext}
\newtheorem{theorem}{Theorem}[section]
\newtheorem{lemma}[theorem]{Lemma}

\numberwithin{equation}{section}

\begin{document}
\setcounter{page}{0}
\thispagestyle{empty}

\begin{center}
{\Large \bf Analysis of the Expected Number of Bit Comparisons\\
Required by Quickselect \\}
\normalsize

\vspace{4ex}
{\sc James Allen Fill\footnotemark} \\
\vspace{.1in}
Department of Applied Mathematics and Statistics \\
\vspace{.1in}
The Johns Hopkins University \\
\vspace{.1in}
{\tt jimfill@jhu.edu} and {\tt http://www.ams.jhu.edu/\~{}fill/} \\
\vspace{.2in}
{\sc and} \\
{\sc Take Nakama}\\
\vspace{.1in}
Department of Applied Mathematics and Statistics \\
\vspace{.1in}
The Johns Hopkins University \\
\vspace{.1in}
{\tt nakama@jhu.edu} and {\tt http://www.ams.jhu.edu/\~{}nakama/} \\
\end{center}
\vspace{3ex}

\begin{center}
{\sl ABSTRACT} \\
\end{center}

When algorithms for sorting and searching
are applied to keys that are represented as bit strings, we can
quantify the performance of the algorithms not only in terms of the
number of key comparisons required by the algorithms but also in
terms of the number of bit comparisons. Some of the standard sorting
and searching algorithms have been analyzed with respect to key
comparisons but not with respect to bit comparisons. In this paper,
we investigate the expected number of bit comparisons required by
{\tt Quickselect} (also known as {\tt Find}). We develop exact and asymptotic
formulae for the expected number of bit comparisons required to find
the smallest or largest key by {\tt Quickselect} and show that the
expectation is asymptotically linear with respect to the number of
keys. Similar results are obtained for the average case.  For
finding keys of arbitrary rank, we derive an exact formula for the
expected number of bit comparisons that (using rational arithmetic)
requires only finite summation (rather than such operations as
numerical integration) and use it to compute the expectation for
each target rank.

\bigskip
\bigskip

\begin{small}

\par\noindent
{\em AMS\/} 2000 {\em subject classifications.\/}  Primary 68W40;
secondary 68P10, 60C05.
%
\medskip
\par\noindent
{\em Key words and phrases.\/}
{\tt Quickselect}, {\tt Find}, searching algorithms, asymptotics, average-case analysis, key comparisons, bit comparisons.
\medskip
\par\noindent
\emph{Date.} June~15, 2007.
\end{small}

\footnotetext[1]{Research for both authors supported by NSF grant DMS--0406104,
and by The Johns Hopkins University's Acheson J.~Duncan Fund for the
Advancement of Research in Statistics.}

\newpage

\section{Introduction and Summary}\label{introduction}
\setlength\parindent{.5in}
When an algorithm for sorting or searching is analyzed, the
algorithm is usually regarded either as comparing keys pairwise
irrespective of the keys' internal structure or as operating on
representations (such as bit strings) of keys.  In the former case,
analyses often quantify the performance of the algorithm in terms of
the number of key comparisons required to accomplish the task;
{\tt Quickselect} (also known as {\tt Find}) is an example of those algorithms
that have been studied from this point of view.  In the latter case,
if keys are represented as bit strings, then analyses quantify the
performance of the algorithm in terms of the number of bits compared
until it completes its task.  Digital search trees, for example,
have been examined from this
perspective.\\
\indent     In order to fully quantify the performance of a sorting
or searching algorithm and enable comparison between key-based and
digital algorithms, it is ideal to analyze the algorithm from both
points of view. However, to date, only {\tt Quicksort} has been analyzed
with both approaches; see Fill and Janson \cite{fj2004}. Before
their study, {\tt Quicksort} had been extensively examined with regard to
the number of key comparisons performed by the algorithm (e.g.,
Knuth \cite{k2002v3}, R\'{e}gnier \cite{r1989}, R\"{o}sler
\cite{r1991}, Knessl and Szpankowski \cite{ks1999}, Fill and Janson
\cite{fj2002}, Neininger and R\"{u}schendorf \cite{nr2002}), but it
had not been examined with regard to the number of bit comparisons
in sorting keys represented as bit strings. In their study, Fill and
Janson assumed that keys are independently and uniformly distributed
over (0,1) and that the keys are represented as bit strings. [They
also conducted the analysis for a general absolutely continuous
distribution over (0,1).]  They showed that the expected number of
bit comparisons required to sort $n$ keys is asymptotically
equivalent to $n(\ln n)(\lg n)$ as compared to the lead-order term
of the expected number of \emph{key} comparisons, which is
asymptotically $2n \ln n$. We use ln and lg to denote natural and
binary logarithms, respectively, and use log when the base does not
matter (for example, in remainder estimates).\\
\indent     In this paper, we investigate the expected number of bit
comparisons required by {\tt Quickselect}. Hoare \cite{h1961} introduced
this search algorithm, which is treated in most textbooks on
algorithms and data structures. {\tt Quickselect} selects the $m$-th
smallest key (we call it the rank-$m$ key) from a set of $n$
distinct keys. (The keys are typically assumed to be distinct, but
the algorithm still works---with a minor adjustment---even if they
are not distinct.) The algorithm finds the target key in a recursive
and random fashion. First, it selects a pivot uniformly at random
from $n$ keys.  Let $k$ denote the rank of the pivot.  If $k=m$,
then the algorithm returns the pivot.  If $k > m$, then the
algorithm recursively operates on the set of keys smaller than the
pivot and returns the rank-$m$ key. Similarly, if $k < m$, then the
algorithm recursively operates on the set of keys larger than the
pivot and returns the ($k-m$)-th smallest key from the subset.
Although previous studies (e.g., Knuth \cite{k1972}, Mahmoud
\textit{et al.} \cite{mms1995}, Gr\"{u}bel and U. R\"{o}sler
\cite{gr1996}, Lend and Mahmoud \cite{lm1996}, Mahmoud and Smythe
\cite{ms1998}, Devroye \cite{d2001}, Hwang and Tsai \cite{ht2002})
examined {\tt Quickselect} with regard to key comparisons, this study is
the first to analyze
the bit complexity of the algorithm.\\
\indent We suppose that the algorithm is applied to $n$ distinct
keys that are represented as bit strings and that the algorithm
operates on individual bits in order to find a target key. We also
assume that the $n$ keys are uniformly and independently distributed
in $(0, 1)$. For instance, consider applying {\tt Quickselect} to find the
smallest key among three keys $k_1$, $k_2$, and $k_3$ whose binary
representations are .01001100..., .00110101..., and .00101010...,
respectively. If the algorithm selects $k_3$ as a pivot, then it
compares each of $k_1$ and $k_2$ to $k_3$ in order to determine the
rank of $k_3$.  When $k_1$ and $k_3$ are compared, the algorithm
requires 2 bit comparisons to determine that $k_3$ is smaller than
$k_1$ because the two keys have the same first digit and differ at
the second digit. Similarly, when $k_2$ and $k_3$ are compared, the
algorithm requires 4 bit comparisons to determine that $k_3$ is
smaller than $k_2$.  After these comparisons, key $k_3$ has been
identified as smallest. Hence the search for the smallest key
requires a total of 6 bit comparisons (resulting from the two key
comparisons).\\
\indent     We let $\mu(m,n)$ denote the expected number of bit
comparisons required to find the rank-$m$ key in a file of $n$ keys
by {\tt Quickselect}. By symmetry, $\mu(m,n)=\mu(n+1-m,n)$. First, we
develop exact and asymptotic formulae for $\mu(1,n)=\mu(n,n)$, the
expected number of bit comparisons required to find the smallest key
by {\tt Quickselect}, as summarized in the following theorem.
\begin{theorem}\label{thmForMu1}
The expected number $\mu (1,n)$ of bit comparisons required by
{\tt Quickselect} to find the smallest key in a file of $n$ keys that are
independently and uniformly distributed in $(0,1)$ has the following
exact and asymptotic expressions:
\begin{eqnarray}
    \mu (1,n) &=& 2n(H_n-1) + 2
    \sum_{j=2}^{n-1}B_j\frac{n-j+1-{{n}\choose{j}}}{j(j-1)(1-2^{-j})}\nonumber\\
    &=& cn-\frac{1}{\ln2}(\ln n)^2-\left(\frac{2}{\ln 2} +
1\right)\ln n + O(1),\nonumber
\end{eqnarray}
where $H_n$ and $B_j$ denote harmonic and Bernoulli numbers,
respectively, and, with $\chi_k := \frac{2 \pi i k}{\ln 2}$ and
$\gamma:=\text{Euler's constant} \doteq 0.57722$, we define
\begin{eqnarray*}
    c := \frac{28}{9} + \frac{17-6\gamma}{9\ln 2} - \frac{4}{\ln
2}\sum_{k \in
\mathbb{Z}\backslash\{0\}}\frac{\zeta(1-\chi_k)\Gamma(1-\chi_k)}{\Gamma(4-\chi_k)(1-\chi_k)}\doteq5.27938.
\end{eqnarray*}
\end{theorem}
The asymptotic formula shows that the expected number of bit
comparisons is asymptotically linear in $n$ with the lead-order
coefficient approximately equal to 5.27938.  Hence the expected
number of bit comparisons is asymptotically different from that of
\emph{key} comparisons required to find the smallest key only by a
constant factor (the expectation for key comparisons is
asymptotically 2$n$). Complex-analytical methods are utilized to
obtain the asymptotic formula. Details of the derivations of the
formulae are
described in Section \ref{analysisOfMu1}.\\
\indent     We also derive exact and asymptotic expressions for the
expected number of bit comparisons for the average case.  We denote
this expectation by $\mu (\bar{m},n)$.  In the average case, the
parameter $m$ in $\mu(m,n)$ is considered a discrete uniform random
variable; hence $\mu (\bar{m},n) = \frac{1}{n} \sum_{m=1}^{n} \mu
(m,n).$ The derived asymptotic formula shows that $\mu (\bar{m},n)$
is also asymptotically linear in $n$; see
(\ref{asymptoticFormulaForMuAvg}).  More detailed results
for $\mu(\bar{m},n)$ are described in Section \ref{analysisOfMuAvg}.\\
\indent    Lastly, in Section \ref{closedForm}, we derive an exact
expression of $\mu (m,n)$ for each fixed $m$ that is suited for
computations. Our preliminary exact formula for $\mu (m,n)$ [shown
in (\ref{eqnMu})] entails infinite summation and integration. As a
result, it is not a desirable form for numerically computing the
expected number of bit comparisons. Hence we establish another exact
formula that only requires finite summation and use it to compute
$\mu (m,n)$ for $m = 1,\ldots,n$, $n=2,\ldots,25$. The computation
leads to the following conjectures: (i) for fixed $n$, $\mu (m,n)$
increases in $m$ for $m \leq \frac{n+1}{2}$ and is symmetric about
$\frac{n+1}{2}$; and (ii) for fixed $m$, $\mu (m,n)$ increases in
$n$ (asymptotically linearly).

\section{Preliminaries}\label{preliminaries}
\setlength\parindent{.5in}
To investigate the bit complexity of {\tt Quickselect}, we follow the
general approach developed by Fill and Janson \cite{fj2004}. Let
$U_1,\ldots,U_n$ denote the $n$ keys uniformly and independently
distributed on (0, 1), and let $U_{(i)}$ denote the rank-$i$ key.
Then, for $1 \leq i < j \leq n$ (assume $n \geq 2$),
\begin{eqnarray}
P\{U_{(i)}\ \mbox{and}\ U_{(j)}\ \mbox{are}\ \mbox{compared}\}
 = \left \{ \begin{array}{ccc}
                \displaystyle \frac{2}{j-m+1} \ \ \ & \mbox{if $m \leq i$} \\
                \displaystyle \frac{2}{j-i+1} \ \ \ & \mbox{if $i < m < j$} \\
                \displaystyle \frac{2}{m-i+1} \ \ \ & \mbox{if $j \leq m$}.
                \end{array}
        \right.\label{event1}
\end{eqnarray}
To determine the first probability in (\ref{event1}), note that
$U_{(m)},\ldots,U_{(j)}$ remain in the same subset until the first
time that one of them is chosen as a pivot. Therefore, $U_{(i)}$ and
$U_{(j)}$ are compared if and only if the first pivot chosen from
$U_{(m)},\ldots,U_{(j)}$ is either $U_{(i)}$ or $U_{(j)}$. Analogous
arguments establish the other two cases.\\
\indent     For $0 < s < t < 1$, it is well known that the joint
density function of $U_{(i)}$ and $U_{(j)}$ is given by
\begin{eqnarray}
f_{U_{(i)},U_{(j)}}(s,t) & := &{n\choose{i-1,1,j-i-1,1,n-j}}\
s^{i-1} (t-s)^{j-i-1} (1-t)^{n-j}.\label{jointDensity}
\end{eqnarray}
Clearly, the event that $U_{(i)}$ and $U_{(j)}$ are compared is
independent of the random variables $U_{(i)}$ and $U_{(j)}$. Hence,
defining
\begin{eqnarray}
P_1(s,t,m,n)  & = & \sum_{m \leq i < j \leq n} \frac{2}{j-m+1} f_{U_{(i)},U_{(j)}}(s,t),\label{eqnP1}\\
P_2(s,t,m,n)  & = & \sum_{1 \leq i < m < j \leq n} \frac{2}{j-i+1} f_{U_{(i)},U_{(j)}}(s,t),\label{eqnP2} \\
P_3(s,t,m,n)  & = & \sum_{1 \leq i < j \leq m} \frac{2}{m-i+1}
f_{U_{(i)},U_{(j)}}(s,t),\label{eqnP3}\\
P(s,t,m,n) & = & P_1(s,t,m,n) + P_2(s,t,m,n) +
P_3(s,t,m,n)\label{eqnP}
\end{eqnarray}
[the sums in (\ref{eqnP1})--(\ref{eqnP3}) are double sums over $i$
and $j$], and letting $\beta(s,t)$ denote the index of the first bit
at which the keys $s$ and $t$ differ, we can write the expectation
$\mu (m,n)$ of the number of bit comparisons required to find the
rank-$m$ key in a file of $n$ keys as
\begin{eqnarray}
\mu (m,n) & = & \int_{0}^{1} \! \int_{s}^{1} \! \beta (s,t) P(s,t,m,n)\, dt\, ds \label{anotherEqnMu}\\
    & = & \sum_{k=0}^{\infty} \sum_{l=1}^{2^k} \int_{(l-1)2^{-k}}^{(l-\frac{1}{2})2^{-k}} \! \int_{(l-\frac{1}{2})2^{-k}}^{l2^{-k}}
    (k+1) P(s,t,m,n)\, dt\, ds;\label{eqnMu}
\end{eqnarray}
in this expression, note that $k$ represents the last bit at which $s$ and $t$ agree.\\

\section{Analysis of $\mu(1,\lowercase{n})$}\label{analysisOfMu1}
\setlength\parindent{.5in}
In Section \ref{derivationOfMu1}, we derive the exact expression for
$\mu(1,n)$ shown in Theorem \ref{thmForMu1}.  In Section
\ref{asymptoticAnalysisOfMu1}, we prove the asymptotic result stated
in Theorem \ref{thmForMu1}.
\subsection{Exact Computation of $\mu(1,n)$}\label{derivationOfMu1}

Since the contribution of $P_2(s,t,m,n)$ or $P_3(s,t,m,n)$ to
$P(s,t,m,n)$ is zero for $m=1$, we have $P(s,t,1,n) = P_1(s,t,1,n)$
[see (\ref{eqnP2}) through (\ref{eqnP})]. Let $x := s,\ y := t-s,\
z:=1-t$. Then
\begin{eqnarray}
P_1(s,t,1,n) & = & z^n \sum_{1 \leq i < j \leq n} \frac{2}{j} {n\choose{i-1,1,j-i-1,1,n-j}}\ x^{i-1} y^{j-i-1} z^{-j}\nonumber\\
    & = & 2z^n \int_{z}^{\infty}\! \eta ^{-n-1} \sum_{1 \leq i < j \leq n} {n\choose{i-1,1,j-i-1,1,n-j}}\ x^{i-1} y^{j-i-1} \eta^{n-j}\, d \eta\nonumber\\
    & = & 2z^n \int_{z}^{\infty}\! \eta ^{-n-1} n(n-1)(x+y+\eta)^{n-2}\, d \eta \nonumber\\
    & = & 2z^n n(n-1) \int_{z}^{\infty}\! \eta ^{-3} \left(\frac{t}{\eta}+1 \right)^{n-2}\, d
    \eta.\label{firstPstM1}
\end{eqnarray}
Making the change of variables $v=\frac{t}{\eta}+1$ and integrating,
and recalling $z=1-t$, we find, after some calculation,
\begin{eqnarray}
P_1(s,t,1,n) & = & 2 \sum_{j=2}^{n} (-1)^j {n \choose j}
t^{j-2}.\label{eqnPstM1}
\end{eqnarray}
From (\ref{eqnMu}) and (\ref{eqnPstM1}),
\begin{eqnarray}
\mu (1,n) & = & \sum_{k=0}^{\infty} (k+1) \sum_{l=1}^{2^k} \int_{(l-1)2^{-k}}^{(l-\mbox{$\frac{1}{2}$})2^{-k}}\! \int_{(l-\mbox{$\frac{1}{2}$})2^{-k}}^{l2^{-k}} \! P_1(s,t,1,n)\, dt\, ds \nonumber \\
    & = & 2 \sum_{k=0}^{\infty} (k+1) \sum_{l=1}^{2^k} \int_{(l-1)2^{-k}}^{(l-\mbox{$\frac{1}{2}$})2^{-k}}\! \int_{(l-\mbox{$\frac{1}{2}$})2^{-k}}^{(l-1)2^{-k}}
    \! \sum_{j=2}^{n} (-1)^j {n \choose j} t^{j-2} \, dt\, ds\nonumber\\
    & = & 2 \sum_{k=0}^{\infty} (k+1) \sum_{l=1}^{2^k} \sum_{j=2}^{n} (-1)^j {n \choose j} \int_{(l-\mbox{$\frac{1}{2}$})2^{-k}}^{l2^{-k}}
    \!  t^{j-2} [(l-\mbox{$\frac{1}{2}$})2^{-k}-(l-1)2^{-k}]\, dt\nonumber\\
    & = & \sum_{k=0}^{\infty} (k+1) \sum_{l=1}^{2^k} \sum_{j=2}^{n} \frac{(-1)^j {n \choose j}}{j-1} 2^{-k}\{(l2^{-k})^{j-1}-[(l-\mbox{$\frac{1}{2}$})2^{-k}]^{j-1}\}\nonumber\\
    & = & \sum_{j=2}^{n} \frac{(-1)^j {n \choose j}}{j-1} \sum_{k=0}^{\infty} (k+1) 2^{-kj}\sum_{l=1}^{2^k}
    [l^{j-1}-(l-\mbox{$\frac{1}{2}$})^{j-1}].\label{intermediateMu1}
\end{eqnarray}
To further transform (\ref{intermediateMu1}), define
\begin{eqnarray}\label{a_jr}
a_{j,r}
 =   \left \{ \begin{array}{ccc}
                \displaystyle \frac{B_r}{r} {{j-1}\choose{r-1}} \ \ \ & \mbox{if $r \geq 2$} \\
                \displaystyle \mbox{$\frac{1}{2}$} \ \ \ & \mbox{if $r=1$} \\
                \displaystyle \mbox{$\frac{1}{j}$} \ \ \ & \mbox{if $r=0$,}
                \end{array}
        \right.
\end{eqnarray}
where $B_r$ denotes the $r$-th Bernoulli number. Let
$S_{n,j}:=\sum_{l=1}^{n}l^{j-1}$.  Then $S_{n,j} = \sum_{r=0}^{j-1}
a_{j,r}n^{j-r}$ (see Knuth \cite{k2002v3}), and
\begin{eqnarray}
\lefteqn{\sum_{l=1}^{2^k}
    [l^{j-1}-(l-\mbox{$\frac{1}{2}$})^{j-1}] = S_{2^k,j} - 2^{-(j-1)}\sum_{l=1}^{2^k}(2l-1)^{j-1}}\nonumber\\
    &=& S_{2^k,j} - 2^{-(j-1)}(S_{2^{k+1},j}-2^{j-1}S_{2^k,j}) = 2S_{2^k,j} - 2^{-(j-1)}S_{2^{k+1},j}\nonumber\\
    &=& 2\sum_{r=0}^{j-1}a_{j,r}2^{k(j-r)}-2^{-(j-1)} \sum_{r=0}^{j-1}a_{j,r}2^{(k+1)(j-r)}
    = 2\sum_{r=1}^{j-1}a_{j,r}2^{k(j-r)}(1-2^{-r}).\label{partOfMu1WithAnj}
\end{eqnarray}
From (\ref{intermediateMu1}) and (\ref{partOfMu1WithAnj}),
\begin{eqnarray}
\mu (1,n)
    & = & 2\sum_{j=2}^{n} \frac{(-1)^j {n \choose j}}{j-1} \sum_{k=0}^{\infty} (k+1) 2^{-kj}
    \sum_{r=1}^{j-1}a_{j,r}2^{k(j-r)}(1-2^{-r}).\nonumber
\end{eqnarray}
Here
\begin{eqnarray}
\lefteqn{\sum_{k=0}^{\infty} (k+1) 2^{-kj}
\sum_{r=1}^{j-1}a_{j,r}2^{k(j-r)}(1-2^{-r})
    = \sum_{k=0}^{\infty} (k+1) \sum_{r=1}^{j-1}a_{j,r}2^{-kr}(1-2^{-r})}\nonumber\\
    &=& \sum_{r=1}^{j-1} a_{j,r} (1-2^{-r})\sum_{k=0}^{\infty} (k+1) 2^{-kr}
    = \sum_{r=1}^{j-1} a_{j,r} (1-2^{-r})^{-1}.\nonumber
\end{eqnarray}
Hence
\begin{eqnarray}
\mu (1,n)
    & = & 2\sum_{j=2}^{n} \frac{(-1)^j {n \choose j}}{j-1} \sum_{r=1}^{j-1} a_{j,r}
    (1-2^{-r})^{-1}
     = 2\sum_{r=1}^{n-1} (1-2^{-r})^{-1} \sum_{j=r+1}^{n} \frac{(-1)^j {n \choose j}}{j-1}  a_{j,r}
    \nonumber\\
    & = & 2\sum_{j=2}^{n} \frac{(-1)^j {n \choose j}}{j-1} +
    2\sum_{r=2}^{n-1} (1-2^{-r})^{-1} \frac{B_r}{r}\sum_{j=r+1}^{n} \frac{(-1)^j {n \choose j}{{j-1} \choose
    {r-1}}}{j-1}\nonumber\\
    &=& 2\sum_{j=2}^{n} \frac{(-1)^j {n \choose j}}{j-1} + 2\sum_{r=2}^{n-1} (1-2^{-r})^{-1} \frac{B_r}{r} \left[ \sum_{j=r}^{n}
    \frac{(-1)^j {n \choose j}{{j-1} \choose {r-1}}}{j-1} - \frac{(-1)^r {n \choose
    r}}{r-1}\right]\!\!.\nonumber\\
    & &\label{intermediateMu1_1}
\end{eqnarray}
To simplify $\sum_{j=r}^{n}\frac{(-1)^j {n \choose j}{{j-1} \choose
{r-1}}}{j-1}$, note that
\begin{eqnarray}
\lefteqn{\sum_{j=r}^{n}{n \choose j}{{j-1} \choose {r-1}}z^{j-2}
    =\frac{r\, n!}{(n-r)!r!}\sum_{j=r}^{n}\frac{(n-r)!}{j(n-j)!(j-r)!}z^{j-2}}\nonumber\\
    &=&r{n \choose r}z^{-2}\sum_{j=r}^{n}{{n-r} \choose {j-r}}
    \frac{z^j}{j}
    =r{n \choose r}z^{-2}\sum_{j=0}^{n-r}{{n-r} \choose {j}}
    \frac{z^{j+r}}{j+r}\nonumber\\
    &=&r{n \choose r} z^{-2} \int_{0}^{z} \zeta^{r-1} \sum_{j=0}^{n-r}{{n-r} \choose {j}}
    \zeta^j\, d \zeta
    =r{n \choose r} z^{-2} \int_{0}^{z} \zeta^{r-1} (1+\zeta)^{n-r}\, d \zeta.\nonumber
\end{eqnarray}
Thus
\begin{eqnarray}
\lefteqn{\sum_{j=r}^{n}\frac{(-1)^j {n \choose j}{{j-1} \choose
    {r-1}}}{j-1}
    = \int_{-1}^{0}\! \left[\sum_{j=r}^{n}{n \choose j}{{j-1} \choose {r-1}}z^{j-2}
    \right]\, dz}\nonumber\\
    &=& -r{n \choose r} \int_{-1}^{0}\! z^{-2} \int_{z}^{0}\! \zeta^{r-1} (1+\zeta)^{n-r}\, d
    \zeta \, dz
    = -r{n \choose r} \int_{-1}^{0}\! \zeta^{r-1} (1+\zeta)^{n-r} \int_{-1}^{\zeta}\! z^{-2}
    \, dz \,d\zeta\nonumber\\
    &=& r{n \choose r} \int_{-1}^{0}\! \zeta^{r-2} (1+\zeta)^{n-r+1}\, d\zeta
    = (-1)^r r{n \choose r} \int_{0}^{1} \! u^{r-2} (1-u)^{n-r+1}\, du\nonumber\\
    &=& (-1)^r r{n \choose r}
    \frac{\Gamma(r-1)\Gamma(n-r+2)}{\Gamma(n+1)}
    = \frac{(-1)^r (n-r+1)}{r-1}.\label{partOfMu1}
\end{eqnarray}
Plugging (\ref{partOfMu1}) into (\ref{intermediateMu1_1}) and
recalling $B_{2k+1}=0$ for $k \geq 1$, we finally obtain
\begin{eqnarray}
\mu (1,n)
    &=& 2\sum_{j=2}^{n} \frac{(-1)^j {n \choose j}}{j-1} + 2\sum_{r=2}^{n-1} (1-2^{-r})^{-1} \frac{B_r}{r} \left[\frac{(-1)^r (n-r+1)}{r-1} - \frac{(-1)^r {n \choose
    r}}{r-1}\right]\nonumber\\
    & = &2 \sum_{j=2}^n \frac{(-1)^j  {{n}\choose{j}}}{j-1} + 2
    \sum_{j=2}^{n-1}B_j\frac{n-j+1-{{n}\choose{j}}}{j(j-1)(1-2^{-j})}\nonumber\\
    & = & 2n(H_n-1)+ 2 t_n \label{eqnMu1},
\end{eqnarray}
where $H_n$ denotes the $n$-th harmonic number and
\begin{eqnarray}
t_n := \sum_{j=2}^{n-1} \frac{B_j}{j(1-2^{-j})} \left[ \frac{n-{{n}
\choose {j}}}{j-1}-1 \right].\label{t_n1}
\end{eqnarray}
The last equality in (\ref{eqnMu1}) follows from the easy identity
\begin{eqnarray}
\sum_{k=1}^{n}\frac{(-1)^{k-1}{{n}\choose{k}}}{k} = H_n.\nonumber
\end{eqnarray}
\newpage
\subsection{Asymptotic Analysis of
$\mu(1,n)$}\label{asymptoticAnalysisOfMu1} In order to obtain an
asymptotic expression for $\mu (1,n)$, we analyze $t_n$ in
(\ref{eqnMu1})--(\ref{t_n1}).  The following lemma provides an exact
expression for $t_n$ that easily leads to an asymptotic expression
for $\mu(1,n)$:
\begin{lemma}\label{lemmaFor_t_n}
For $n\geq 2$, let $u_n:=t_{n+1} - t_{n}$ $(with\ t_2 = 0)$ and $v_n
:= v_{n+1}-v_{n}$. Let $\gamma$ denote Euler's constant $(\doteq
0.57722)$, and define $\chi_k := \frac{2\pi ik}{\ln 2}$. Then
\begin{enumerate}
\item[$(i)$]
\begin{eqnarray}v_n = \frac{1}{n+1}+\frac{\frac{H_{n+2}}{\ln
2}-(\frac{\gamma}{\ln
2}-\frac{1}{2})}{(n+1)(n+2)}-\Sigma_n,\nonumber
\end{eqnarray}
where
\begin{eqnarray}
\Sigma_n := \sum_{k \in
\mathbb{Z}\backslash\{0\}}\frac{\zeta(1-\chi_k) \Gamma(n+1)
\Gamma(1-\chi_k)}{(\ln 2) \Gamma(n+3-\chi_k)};\nonumber
\end{eqnarray}
\item[$(ii)$]
\begin{eqnarray}
u_n = -H_n + a - \frac{H_{n+1}}{(\ln2)(n+1)}+\left(\frac{\gamma
-1}{\ln2}-\frac{1}{2}\right)\frac{1}{n+1}+\tilde{\Sigma}_n,\nonumber
\end{eqnarray}
where
\begin{eqnarray} a &:=&
\frac{14}{9}+\frac{17-6\gamma}{18\ln2}-\frac{2}{\ln 2}\sum_{k \in
\mathbb{Z} \backslash
\{0\}}\frac{\zeta(1-\chi_k)\Gamma(1-\chi_k)}{\Gamma(4-\chi_k)(1-\chi_k)},\nonumber\\
\tilde{\Sigma}_n &:=& \sum_{k \in Z \setminus
\{0\}}\frac{\zeta(1-\chi_k)\Gamma(1-\chi_k)}{(\ln2)(1-\chi_k)}
\frac{\Gamma(n+1)}{\Gamma(n+2-\chi_k)};\nonumber
\end{eqnarray}
\item[$(iii)$]
\begin{eqnarray}
t_n &=& -(nH_n-n-1) + a(n-2) -\frac{1}{2\ln2}\left[
H_n^2 + H_n^{(2)} - \frac{7}{2}\right]\nonumber\\
    & & +\left(\frac{\gamma -1}{\ln2}-\frac{1}{2}\right)\left( H_n -
\frac{3}{2}\right) + b - \tilde{\tilde{\Sigma}}_n,\nonumber
\end{eqnarray}
where
\begin{eqnarray}
b &:=& \sum_{k \in
\mathbb{Z}\backslash\{0\}}\frac{2\zeta(1-\chi_k)\Gamma(-\chi_k)}{(\ln2)(1-\chi_k)\Gamma(3-\chi_k)},\nonumber\\
\tilde{\tilde{\Sigma}}_n &:=& \sum_{k \in
\mathbb{Z}\backslash\{0\}}\frac{\zeta(1-\chi_k)\Gamma(-\chi_k)\Gamma(n+1)}{(\ln2)(1-\chi_k)\Gamma(n+1-\chi_k)},\nonumber
\end{eqnarray}
and $H_n^{(2)}$ denotes the $n$-th Harmonic number of order 2, i.e.,
$H_n^{(2)} := \sum_{i=1}^{n}\frac{1}{i^2}$.
\end{enumerate}
\end{lemma}
\noindent In this lemma, $u_n$ and $v_n$ are derived in order to
obtain the exact expression for $t_n$ in (iii). From
($\ref{eqnMu1}$), the exact expression for $t_n$ also provides an
alternative exact expression for $\mu(1,n)$.\\
\indent Before proving Lemma \ref{lemmaFor_t_n}, we complete the
proof of Theorem \ref{thmForMu1} using part (iii).  We know
\begin{eqnarray}
H_n & = & \ln n + \gamma + \frac{1}{2n} -
\frac{1}{12n^2}+O(n^{-4}),\label{H_n}\\
H_n^{(2)} & = &
\frac{\pi^2}{6}-\frac{1}{n}+\frac{1}{2n^2}+O(n^{-3}).\label{H2_n}
\end{eqnarray}
Combining (\ref{H_n})--(\ref{H2_n}) with (\ref{eqnMu1}) and Lemma
\ref{lemmaFor_t_n}(iii), we obtain an asymptotic expression for $\mu
(1,n)$:
\begin{eqnarray}
\mu (1,n) & = & 2an-\frac{1}{\ln2}(\ln n)^2-\left(\frac{2}{\ln 2} +
1\right)\ln n + O(1).\label{asymptM1}
\end{eqnarray}
The term $O(1)$ in (\ref{asymptM1}) has fluctuations of small
magnitude due to $\tilde{\tilde{\Sigma}}_n$, which is periodic in
$\log n$ with amplitude smaller than 0.00110. The asymptotic slope
in (\ref{asymptM1}) is
\begin{eqnarray}
c = 2a = \frac{28}{9} + \frac{17-6\gamma}{9\ln 2} - \frac{4}{\ln
2}\sum_{k \in
\mathbb{Z}\backslash\{0\}}\frac{\zeta(1-\chi_k)\Gamma(1-\chi_k)}{\Gamma(4-\chi_k)(1-\chi_k)}
\doteq 5.27938.\label{asymptoticSlopeForMu1}
\end{eqnarray}

Now we prove Lemma \ref{lemmaFor_t_n}:
\begin{proof}
(i) Since
\begin{eqnarray}
u_n & = & t_{n+1} - t_{n} = \sum_{j=2}^{n} \frac{B_j}{j(1-2^{-j})}
\left[ \frac{(n+1)-{{n+1} \choose {j}}}{j-1}-1 \right] -
\sum_{j=2}^{n-1} \frac{B_j}{j(1-2^{-j})} \left[ \frac{n-{{n} \choose
{j}}}{j-1}-1
\right] \nonumber \\
& = & -\sum_{j=2}^{n} \frac{B_j}{j(j-1)(1-2^{-j})}
\left[{{n}\choose{j-1}}-1\right],\nonumber
\end{eqnarray}
it follows that
\begin{eqnarray}
v_n &=& u_{n+1}-u_n = -\sum_{j=2}^{n+1} \frac{B_j}{j(j-1)(1-2^{-j})}
\left[{{n+1}\choose{j-1}}-1\right] + \sum_{j=2}^{n+1}
\frac{B_j}{j(j-1)(1-2^{-j})}
\left[{{n}\choose{j-1}}-1\right]\nonumber\\
& = & -\sum_{k=0}^{n-1}{{n}\choose{k}}
\frac{B_{k+2}}{(k+2)(k+1)[1-2^{-(k+2)}]}\nonumber\\
& = & \sum_{k=0}^{n-1}(-1)^k {{n}\choose{k}}
\frac{\zeta(-1-k)}{(k+1)[1-2^{-(k+2)}]}\label{BernoulliAndZeta}\\
& = & \frac{(-1)^n}{2\pi i}\int_{\mathcal{C}} \!
\frac{\zeta(-1-s)}{(s+1)[1-2^{-(s+2)}]}\frac{n!}{s(s-1)\cdots(s-n)}\,
ds \label{v_nIntegralRepresenation},
\end{eqnarray}
where $\mathcal{C}$ is a positively oriented closed curve that
encircles the integers 0,$\dots$, $n-1$ and does not include or
encircle any of the following points: $-2+\chi_k$ (where $\chi_k :=
\frac{2 \pi i k}{\ln 2}$), $k \in \mathbb{Z}$; $-1$; and $n$.
Equality (\ref{BernoulliAndZeta}) follows from the fact that the
Bernoulli numbers are extrapolated by the Riemann zeta function
taken at nonnegative integers: $B_k = -k \zeta(1-k)$. [The
coefficients $(-1)^k$ do not concern us since the Bernoulli numbers
of odd index greater than 1 vanish.]  Equality
(\ref{v_nIntegralRepresenation}) follows from a direct application
of residue calculus, taking into account contributions of the simple
poles at the integers 0,$\dots$, $n-1$.\\
\indent Let $\phi (s)$ denote the integrand in
(\ref{v_nIntegralRepresenation}):
\begin{eqnarray}
\phi(s)
=\frac{\zeta(-1-s)}{(s+1)[1-2^{-(s+2)}]}\frac{n!}{s(s-1)\cdots(s-n)}.\nonumber
\end{eqnarray}
We consider a positively oriented rectangular contour
$\mathcal{C}_l$ with horizontal sides $\textrm{Im}(s) = \lambda_l$
and $\textrm{Im} (s) = -\lambda_l$, where $\lambda_l :=
\frac{(2l+1)\pi}{\ln2}$, $l \in \mathbb{Z}^{+}$, and vertical sides
$\textrm{Re} (s) = n - \theta$ and $\textrm{Re} (s) = -\lambda_l$,
where $0 < \theta < 1$.  By elementary bounds on $\phi(s)$ along
$\mathcal{C}_l$ and the fact that
\begin{eqnarray}
\int_{n-\theta-i\infty}^{n-\theta+i\infty} \! \phi(s)\, ds =
0\label{complexIntegralAlongRVL}
\end{eqnarray}
(this is implicit on page 113 of Flajolet and Sedgewick
\cite{fs1995} and explicitly proved in the Appendix), one can show
that
\begin{eqnarray}
\lim_{l \rightarrow \infty} \int_{\mathcal{C}_l} \! \phi(s)\, ds =
0.\nonumber
\end{eqnarray}
Accounting for residues due to the poles encircled by
$\mathcal{C}_l$, we obtain
\begin{eqnarray}
v_n & = & (-1)^{n+1} \left\{\textrm{Res}_{s = -1}[\phi(s)]
 + \, \textrm{Res}_{s = -2}[\phi(s)] + \sum_{k \in
\mathbb{Z}\backslash
\{0\}} \! \textrm{Res}_{s = -2 + \chi_k}[\phi(s)] \right\}\nonumber\\
& = &
-\frac{1}{n+1}+\frac{\frac{H_{n+2}}{\ln2}-(\frac{\gamma}{\ln2}-\mbox{$\frac{1}{2}$})}{(n+1)(n+2)}-\Sigma_n,\label{v_n}
\end{eqnarray}
where
\begin{eqnarray}
\Sigma_n := \sum_{k \in
\mathbb{Z}\backslash\{0\}}\frac{\zeta(1-\chi_k) \Gamma(n+1)
\Gamma(1-\chi_k)}{(\ln 2) \Gamma(n+3-\chi_k)}.\label{sigma}
\end{eqnarray}
\qed

\noindent (ii) We have $u_2 = t_3 - t_2 = t_3 = -\frac{1}{9}$.
Hence, from (i),
\begin{eqnarray}
u_n & = & u_2 + \sum_{j=2}^{n-1} v_j = -\frac{1}{9} + \sum_{j=2}^{n-1} v_j\nonumber\\
& = & -\frac{1}{9} - \sum_{j=2}^{n-1}\frac{1}{j+1} +
\frac{1}{\ln2}\sum_{j=2}^{n-1}\frac{H_{j+2}}{(j+1)(j+2)}-\left(\frac{\gamma}{\ln2}-\frac{1}{2}\right)
\sum_{j=2}^{n-1}\frac{1}{(j+1)(j+2)} - \sum_{j=2}^{n-1}
\Sigma_j\nonumber\\
& = & -\frac{1}{9} - (H_n - H_2) +
\frac{1}{\ln2}\sum_{j=2}^{n-1}\frac{H_{j+2}}{(j+1)(j+2)}-\left(\frac{\gamma}{\ln2}-\frac{1}{2}\right)
\left(\frac{1}{3}-\frac{1}{n+1}\right) - \sum_{j=2}^{n-1}
\Sigma_j\nonumber\\
& = & \frac{14}{9} - \frac{\gamma}{3\ln2}-H_n +
\left(\frac{\gamma}{\ln2}-\frac{1}{2}\right)\frac{1}{n+1}+
\frac{1}{\ln2}\sum_{j=2}^{n-1}\frac{H_{j+2}}{(j+1)(j+2)} -
\sum_{j=2}^{n-1} \Sigma_j.\label{initialU_n}
\end{eqnarray}
Here
\begin{eqnarray}
\sum_{j=2}^{n-1}\frac{H_{j+2}}{(j+1)(j+2)} & = &
\sum_{j=3}^{n}\frac{H_{j+1}}{j} -
\sum_{j=4}^{n+1}\frac{H_{j}}{j}\nonumber\\
& = & \frac{H_4}{3}
+\sum_{j=4}^{n}\frac{H_{j+1}-H_j}{j}-\frac{H_{n+1}}{n+1}\label{nGeq3}\\
& = & \frac{17}{18} - \frac{H_n + 1}{n+1},\label{fineWithN2}
\end{eqnarray}
where we assume $n\geq3$ for (\ref{nGeq3}), but (\ref{fineWithN2})
holds also for $n=2$. In regard to $\sum_{j=2}^{n-1}\Sigma_j$, note
that
\begin{eqnarray}
\Sigma_n = -\sum_{k \in Z\setminus
\{0\}}\frac{\zeta(1-\chi_k)\Gamma(1-\chi_k)}{(\ln2)(1-\chi_k)}
\left[\frac{\Gamma(n+2)}{\Gamma(n+3-\chi_k)}-\frac{\Gamma(n+1)}{\Gamma(n+2-\chi_k)}\right],\nonumber
\end{eqnarray}
so that
\begin{eqnarray}
\sum_{j=2}^{n-1}\Sigma_j = -\sum_{k \in Z \setminus
\{0\}}\frac{\zeta(1-\chi_k)\Gamma(1-\chi_k)}{(\ln2)(1-\chi_k)}
\left[\frac{\Gamma(n+1)}{\Gamma(n+2-\chi_k)}-\frac{\Gamma(3)}{\Gamma(4-\chi_k)}\right].\label{sumOfSigma}
\end{eqnarray}
Define
\begin{eqnarray}
\tilde{\Sigma}_n:=\sum_{k \in Z \setminus
\{0\}}\frac{\zeta(1-\chi_k)\Gamma(1-\chi_k)}{(\ln2)(1-\chi_k)}
\frac{\Gamma(n+1)}{\Gamma(n+2-\chi_k)}.\label{sigma_tilde}
\end{eqnarray}
Then, combining (\ref{initialU_n}), (\ref{fineWithN2}), and
(\ref{sumOfSigma}), we obtain
\begin{eqnarray}
u_n & = & -H_n + a - \frac{H_{n+1}}{(\ln2)(n+1)}+\left(\frac{\gamma
-1}{\ln2}-\frac{1}{2}\right)\frac{1}{n+1}+\tilde{\Sigma}_n,\nonumber
\end{eqnarray}
where
\begin{eqnarray}
a := \frac{14}{9}+\frac{17-6\gamma}{18\ln2}-\frac{2}{\ln 2}\sum_{k
\in \mathbb{Z} \backslash
\{0\}}\frac{\zeta(1-\chi_k)\Gamma(1-\chi_k)}{\Gamma(4-\chi_k)(1-\chi_k)}.\label{a}
\end{eqnarray}
\qed

\noindent (iii) Closely following the derivation of $u_n$
described above, we obtain (for $n\geq2$)
\begin{eqnarray}
t_n & = & t_2 + \sum_{j=2}^{n-1} u_j = \sum_{j=2}^{n-1} u_j\nonumber\\
& = & -\sum_{j=2}^{n-1} H_j + a(n-2)
-\frac{1}{\ln2}\sum_{j=3}^{n}\frac{H_j}{j}+
\left(\frac{\gamma-1}{\ln2}-\frac{1}{2}\right)\left(H_n -
\frac{3}{2}\right) + \sum_{j=2}^{n-1}\tilde{\Sigma}_j\nonumber\\
& = &-(nH_n-n-1) + a(n-2) -\frac{1}{2\ln2}\left[
H_n^2 + H_n^{(2)} - \frac{7}{2}\right]\nonumber\\
&   & +\left(\frac{\gamma -1}{\ln2}-\frac{1}{2}\right)\left( H_n -
\frac{3}{2}\right) + b - \tilde{\tilde{\Sigma}}_n,\label{t_n} 
\end{eqnarray}
where
\begin{eqnarray}
b &:=& \sum_{k \in
\mathbb{Z}\backslash\{0\}}\frac{2\zeta(1-\chi_k)\Gamma(-\chi_k)}{(\ln2)(1-\chi_k)\Gamma(3-\chi_k)},\label{b}\\
\tilde{\tilde{\Sigma}}_n &:=& \sum_{k \in
\mathbb{Z}\backslash\{0\}}\frac{\zeta(1-\chi_k)\Gamma(-\chi_k)\Gamma(n+1)}{(\ln2)(1-\chi_k)\Gamma(n+1-\chi_k)}.\label{sigma_tilde_tilde}
\end{eqnarray}
\end{proof}

\section{Analysis of the Average Case: $\mu(\lowercase{\bar{m}},\lowercase{n})$}\label{analysisOfMuAvg}
\setlength\parindent{.5in}
\subsection{Exact Computation of $\mu(\bar{m},n)$}\label{derivationOfMuAvg}

Here we consider the parameter $m$ in $\mu(m,n)$ as a discrete
random variable with probability mass function $P\{m=i\} =
\displaystyle \frac{1}{n},\ i=1,2,\ldots,n$, and average over $m$
while the parameter $n$ is fixed.  Thus, using the notation defined
in (\ref{eqnP1}) through (\ref{anotherEqnMu}),
\begin{eqnarray}
\mu (\bar{m},n) & = & \frac{1}{n} \sum_{m=1}^{n} \mu (m,n)
     = \frac{1}{n} \sum_{m=1}^{n} \int_{0}^{1}\! \int_{s}^{1} \! \beta (s,t) P(s,t,m,n)\, dt\, ds\nonumber \\
    & = & \int_{0}^{1}\! \int_{s}^{1} \! \beta (s,t)\ \frac{1}{n} \sum_{m=1}^{n}
    P(s,t,m,n)\,
    dt\, ds
     = \mu_1(\bar{m},n) + \mu_2(\bar{m},n) +
    \mu_3(\bar{m},n),\nonumber
\end{eqnarray}
where, for $l=1,2,3,$
\begin{eqnarray}
    \mu_l(\bar{m},n) & = & \int_{0}^{1} \int_{s}^{1} \! \beta (s,t)\ \frac{1}{n} \sum_{m=1}^{n} P_l(s,t,m,n)\ dt\,
    ds.\label{Mu_1_2_3}
\end{eqnarray}
Here $\mu_1(\bar{m},n)=\mu_3(\bar{m},n)$, since
\begin{eqnarray*}
P_3(1-t',1-s',n-m'+1,n) = P_1(s',t',m',n)
\end{eqnarray*}
by an easy symmetric argument we omit, and so
\begin{eqnarray}
    \mu_3(\bar{m},n) & = & \int_{0}^{1} \int_{s}^{1} \! \beta (s,t)\ \frac{1}{n} \sum_{m=1}^{n} P_3(s,t,m,n)\ dt\, ds \nonumber\\
    & = & \int_{0}^{1} \int_{s'}^{1} \! \beta (1-t',1-s')\ \frac{1}{n} \sum_{m'=1}^{n}
    P_3(1-t',1-s',n - m' + 1,n)\ dt'\, ds' \nonumber \\
    & = & \int_{0}^{1} \int_{s'}^{1} \beta (s',t')\ \frac{1}{n} \sum_{m'=1}^{n} P_1(s',t',m',n)\ dt'\, ds' \nonumber \\
    & = & \mu_1(\bar{m},n).\nonumber
\end{eqnarray}
Therefore
\begin{eqnarray}
\mu (\bar{m},n) = 2 \mu_1(\bar{m},n) +
\mu_2(\bar{m},n),\label{eqnMuAvg}
\end{eqnarray}
and we will compute $\mu_1(\bar{m},n)$ and $\mu_2(\bar{m},n)$
exactly in Sections \ref{derivationOfMuAvg1}-2.


\subsubsection{Exact Computation of $\mu_1(\bar{m},n)$}\label{derivationOfMuAvg1}

We use the following lemma in order to compute $\mu_1(\bar{m},n)$
exactly:
\begin{lemma}\label{lemmaForMuAvg1}
\begin{eqnarray*}
\lefteqn{\int_{0}^{1}\!
\int_{s}^{1}\!\beta(s,t)\, \frac{1}{n}\sum_{m=2}^{n}P_1(s,t,m,n)\,dt\,ds}\\
&=&2\sum_{j=2}^{n-1} \frac{(-1)^j  {{n-1}\choose{j}}}{j(j-1)}
    +\frac{2}{9}\sum_{j=2}^{n-1} \frac{(-1)^j  {{n-1}\choose{j}}}{j-1}
    -2\sum_{j=3}^{n-1}B_j\frac{n-j+1-{{n-1}\choose{j-1}}}{j(j-1)(j-2)(1-2^{-j})}\\
    & &-2\sum_{j=2}^{n-1}\frac{(-1)^{j}{{n-1}\choose{j}}}{(j+1)j(j-1)(1-2^{-j})}.
\end{eqnarray*}
\end{lemma}
Before proving the lemma, we complete the computation of
$\mu_1(\bar{m},n)$. Note that
\begin{eqnarray*}
\mu_1(\bar{m},n) &=& \int_{0}^{1} \int_{s}^{1} \! \beta (s,t)\,
\frac{1}{n} \sum_{m=1}^{n} P_1(s,t,m,n)\,dt\, ds\\
&=&\frac{1}{n}\int_{0}^{1} \int_{s}^{1} \! \beta (s,t)\,
P_1(s,t,1,n)\,dt\, ds +\int_{0}^{1} \int_{s}^{1} \! \beta (s,t)\,
\frac{1}{n}
\sum_{m=2}^{n} P_1(s,t,m,n)\,dt\, ds\\
&=&\frac{1}{n}\mu(1,n) + \int_{0}^{1} \int_{s}^{1} \! \beta (s,t)\,
\frac{1}{n} \sum_{m=2}^{n} P_1(s,t,m,n)\,dt\, ds.
\end{eqnarray*}
Therefore, by (\ref{eqnMu1}) and Lemma \ref{lemmaForMuAvg1}, we
obtain
\begin{eqnarray}
\mu_1(\bar{m},n)
    &=& \frac{2}{n}\sum_{j=2}^n \frac{(-1)^j  {{n}\choose{j}}}{j-1} +
    \frac{2}{n}\sum_{j=2}^{n-1}B_j\frac{n-j+1-{{n}\choose{j}}}{j(j-1)(1-2^{-j})}\nonumber\\
    & &+2\sum_{j=2}^{n-1} \frac{(-1)^j  {{n-1}\choose{j}}}{j(j-1)}
    +\frac{2}{9}\sum_{j=2}^{n-1} \frac{(-1)^j  {{n-1}\choose{j}}}{j-1}
    -2\sum_{j=3}^{n-1}B_j\frac{n-j+1-{{n-1}\choose{j-1}}}{j(j-1)(j-2)(1-2^{-j})}\nonumber\\
    & &-2\sum_{j=2}^{n-1}\frac{(-1)^{j}{{n-1}\choose{j}}}{(j+1)j(j-1)(1-2^{-j})}\nonumber\\
    &=& n-1- 4\sum_{j=3}^{n} \frac{(-1)^j
    {{n-1}\choose{j-1}}}{j(j-1)(j-2)} + \frac{2}{n}
    \sum_{j=2}^{n-1}B_j\frac{n-j+1-{{n}\choose{j}}}{j(j-1)(1-2^{-j})}\nonumber\\
    & &
    +\frac{2}{9}\sum_{j=2}^{n-1} \frac{(-1)^j  {{n-1}\choose{j}}}{j-1}
    -2\sum_{j=3}^{n-1}B_j\frac{n-j+1-{{n-1}\choose{j-1}}}{j(j-1)(j-2)(1-2^{-j})}\nonumber\\
    & &-2\sum_{j=2}^{n-1}\frac{(-1)^{j}{{n-1}\choose{j}}}{(j+1)j(j-1)(1-2^{-j})},\label{eqnMu1Avg_2}
\end{eqnarray}
where the second equality holds since
\begin{eqnarray*}
\lefteqn{\frac{2}{n}\sum_{j=2}^n \frac{(-1)^j  {{n}\choose{j}}}{j-1}
+2\sum_{j=2}^{n-1} \frac{(-1)^j
    {{n-1}\choose{j}}}{j(j-1)}}\\
    &=&2\sum_{j=2}^n \frac{(-1)^j(n-1)!}{j!(n-j)!(j-1)} -2\sum_{j=3}^{n}
    \frac{(-1)^j(n-1)!}{(j-1)!(n-j)!(j-1)(j-2)}\\
    &=&n-1+2\sum_{j=3}^{n}
    \frac{(-1)^j(n-1)!}{(j-1)!(n-j)!(j-1)}\left[\frac{1}{j}-\frac{1}{j-2}\right]\\
    &=&n-1-4\sum_{j=3}^{n}
    \frac{(-1)^j{{n-1}\choose{j-1}}}{j(j-1)(j-2)}.
\end{eqnarray*}

In Section \ref{derivationOfMuAvg2AndMuAvg} we combine the
expression for $\mu_1(\bar{m},n)$ in (\ref{eqnMu1Avg_2}) with a
similar expression for $\mu_2(\bar{m},n)$ to obtain an exact
expression for $\mu(\bar{m},n)$. The remainder of this section is
devoted to proving Lemma \ref{lemmaForMuAvg1}. For this, the
following expression for $P_1(s,t,m,n)$ will prove useful:
\begin{lemma}\label{lemmaForP1stmn}
Let $m \geq 2$ and let $x:=s$, $y:=t-s$, $z:=1-t$.  Then the
quantity $P_1(s,t,m,n)$ defined at $(\ref{eqnP1})$ satisfies
\begin{eqnarray}
\lefteqn{P_1(s,t,m,n)}\nonumber\\
    &=&2n\int_{0}^{x}\!\frac{1}{(\xi+y)^2}[\Upsilon_1(m,n,\xi,x,y,z)-\Upsilon_2(m,n,\xi,x,y,z)+\Upsilon_3(m,n,\xi,x,y,z)]\, d\xi,\nonumber\\
    & &\label{Pst1_4}
\end{eqnarray}
where
\begin{eqnarray*}
    \Upsilon_1(m,n,\xi,x,y,z) &:=&
    {{n-1}\choose{m-2}}(x-\xi)^{m-2}(n-m)(\xi+y+z)^{n-m+1},\\
    \Upsilon_2(m,n,\xi,x,y,z) &:=&
    {{n-1}\choose{m-2}}(x-\xi)^{m-2}(n-m+1)z(\xi+y+z)^{n-m},\\
    \Upsilon_3(m,n,\xi,x,y,z) &:=&
    {{n-1}\choose{m-2}}(x-\xi)^{m-2}z^{n-m+1}.
\end{eqnarray*}
\end{lemma}
\begin{proof}[Proof of Lemma \ref{lemmaForP1stmn}]
By (\ref{jointDensity})--(\ref{eqnP1}),
\begin{eqnarray}
\lefteqn{P_1(s,t,m,n) 
     = \sum_{m \leq i < j \leq n} \frac{2}{j-m+1} \frac{n!}{(i-1)!(j-i-1)!(n-j)!}\ x^{i-1} y^{j-i-1} z^{n-j}}\nonumber\\
    & = & \sum_{m \leq i < j \leq n} \frac{2}{j-m+1}\frac{n!}{(n-m-1)!} {{n-m-1}\choose {i-m,j-i-1,n-j}}\frac{(i-m)!}{(i-1)!}\ x^{i-1} y^{j-i-1} z^{n-j}\nonumber\\
    & = & \frac{2\, n!}{(n-m-1)!} \sum_{m \leq i < j \leq n} \frac{1}{j-m+1}{{n-m-1}\choose {i-m,j-i-1,n-j}}\frac{(i-m)!}{(i-1)!}\ x^{i-1} y^{j-i-1}
    z^{n-j}.\nonumber\\
    & & \label{Pst1_1}
\end{eqnarray}
In order to compactly describe the derivation of (\ref{Pst1_4}), we
define the following indefinite integration operator $T$:
\begin{eqnarray*}
T(f(x)):= \int_{0}^{x}\! f(\xi)\,d\xi.
\end{eqnarray*}
We really should write $(Tf)(x)$ rather than $T(f(x))$, but we would
like to use shorthand such as $T(x^j) = \frac{x^{j+1}}{j+1}$ when
$j>-1$.
The operator $T$ treats its argument $f$ as a function of $x$; the
other variables involved in $f$ (namely, $y$ and $z$) are treated as
constants. The notation $T^l$ will denote the $l$-th iterate of $T$.
In this notation, for $m < i$,
\begin{eqnarray*}
    \frac{(i-m)!}{(i-1)!}\ x^{i-1} = T^{m-1}(x^{i-m}),
\end{eqnarray*}
and the sum in (\ref{Pst1_1}) equals
\begin{eqnarray*}
    \mbox{\large $T^{m-1}$}\left(\sum_{m \leq i < j \leq n} \frac{1}{j-m+1}{{n-m-1}\choose
    {i-m,j-i-1,n-j}}\ x^{i-m} y^{j-i-1}
    z^{n-j}\right).
\end{eqnarray*}
Here
\begin{eqnarray}
    \frac{1}{j-m+1}\ z^{n-j} =
    z^{n-m+1}\int_{z}^{\infty}\! \eta^{-(j-m+1)-1} d\eta,\nonumber
\end{eqnarray}
so
\begin{eqnarray}
\lefteqn{\mbox{\large $T^{m-1}$}\left(\sum_{m \leq i < j \leq n}
    \frac{1}{j-m+1}{{n-m-1}\choose
    {i-m,j-i-1,n-j}}\ x^{i-m} y^{j-i-1}
    z^{n-j}\right)}\nonumber\\
    &=& z^{n-m+1}\, \mbox{\large $T^{m-1}$}\left(\!\int_{z}^{\infty}\!\left[\sum_{m \leq i < j \leq n}{{n-m-1}\choose
    {i-m,j-i-1,n-j}}\ x^{i-m} y^{j-i-1} \eta^{-j+m-2} \right] d\eta \right)\nonumber\\
    &=& z^{n-m+1}\, \mbox{\large $T^{m-1}$}\left(\!\int_{z}^{\infty}\!\eta^{-n+m-2}(x+y+\eta)^{n-m-1} d\eta\
    \right)\nonumber\\
    &=& z^{n-m+1}\, \mbox{\large $T^{m-1}$}\left(\!\int_{z}^{\infty}\!\eta^{-3}\left(\frac{t}{\eta}+1\right)^{n-m-1} d\eta \right)\label{partOfPst1forAVG}
\end{eqnarray}
(note that $x+y = t$).  Making the change of variables
$v=\frac{t}{\eta}+1$ and integrating, we obtain, after some
computation,
\begin{eqnarray}
\lefteqn{\int_{z}^{\infty}\!\eta^{-3}\left(\frac{t}{\eta}+1\right)^{n-m-1} d\eta}\nonumber\\
    &=&
    \frac{1}{t^2(n-m+1)(n-m)}\left[(n-m)\left(1+\frac{t}{z}\right)^{n-m+1}-(n-m+1)\left(1+\frac{t}{z}\right)^{n-m}+1\right].\nonumber\\
    & & \label{partOfPst1forAVG_1}
\end{eqnarray}
From (\ref{Pst1_1}) and
(\ref{partOfPst1forAVG})--(\ref{partOfPst1forAVG_1}),
\begin{eqnarray}
\lefteqn{P_1(s,t,m,n)}\nonumber\\
    &=& \frac{2\,n!}{(n-m+1)!}\, \mbox{$T^{m-1}$}\left(t^{-2} [(n-m)(z+t)^{n-m+1}-(n-m+1)z(z+t)^{n-m}+z^{n-m+1}]\right).\nonumber\\
    & & \label{Pst1_2}
\end{eqnarray}
Here
\begin{eqnarray}
t^{-2}[(n-m)(z+t)^{n-m+1}-(n-m+1)z(z+t)^{n-m}+z^{n-m+1}]
    &=&\sum_{r=2}^{n-m+1}t^{r-2}\Upsilon(m,n,r,z),\nonumber\\
    &&\label{partOfPst1forAVG_2}
\end{eqnarray}
where
\begin{eqnarray}
    \Upsilon(m,n,r,z) := (n-m){{n-m+1}\choose {r}}z^{n-m+1-r}-(n-m+1){{n-m}\choose
    {r}}z^{n-m+1-r}.\label{Upsilon}
\end{eqnarray}
Then, since $t=x+y$,
\begin{eqnarray}
    \sum_{r=2}^{n-m+1}t^{r-2}\Upsilon(m,n,r,z)
    =\sum_{r=2}^{n-m+1}\Upsilon(m,n,r,z)\sum_{j=0}^{r-2}{{r-2}\choose
    {j}} x^j y^{r-2-j}.\label{partOfPst1forAVG_3}
\end{eqnarray}
From (\ref{Pst1_2})--(\ref{partOfPst1forAVG_3}),
\begin{eqnarray}
\lefteqn{P_1(s,t,m,n)}\nonumber\\
    &=&
    \frac{2\, n!}{(n-m+1)!}\, \mbox{\large $T^{m-1}$}\left(\sum_{r=2}^{n-m+1}\Upsilon(m,n,r,z)\sum_{j=0}^{r-2}{{r-2}\choose
    {j}} x^j y^{r-2-j}\right)\nonumber\\
    &=&
    \frac{2\, n!}{(n-m+1)!}\sum_{r=2}^{n-m+1}\Upsilon(m,n,r,z)\sum_{j=0}^{r-2}{{r-2}\choose
    {j}}y^{r-2-j} \mbox{$T^{m-1}$}(x^j)\nonumber\\
    &=&
    \frac{2\,n!}{(n-m+1)!}\sum_{r=2}^{n-m+1}\Upsilon(m,n,r,z)\sum_{j=0}^{r-2}{{r-2}\choose
    {j}}y^{r-2-j} \frac{x^{j+m-1}}{(j+1)\cdots (j+m-1)}.\nonumber\\
    & & \label{Pst1_3}
\end{eqnarray}
Because of the partial fraction expansion
\begin{eqnarray}
    \frac{1}{(j+1)\cdots (j+m-1)}
    =\frac{1}{(m-2)!}\sum_{l=0}^{m-2}
    \frac{(-1)^l{{m-2}\choose{l}}}{j+l+1},\nonumber
\end{eqnarray}
it follows that
\begin{eqnarray}
\lefteqn{\sum_{j=0}^{r-2}{{r-2}\choose {j}}y^{r-2-j}
    \frac{x^{j+m-1}}{(j+1)\cdots
    (j+m-1)}}\nonumber\\
    &=&\sum_{j=0}^{r-2}{{r-2}\choose
    {j}}y^{r-2-j}\frac{x^{j+m-1}}{(m-2)!} \sum_{l=0}^{m-2}
    \frac{(-1)^l{{m-2}\choose{l}}}{j+l+1}\nonumber\\
    &=&\frac{1}{(m-2)!} \sum_{l=0}^{m-2}(-1)^l{{m-2}\choose{l}}x^{m-2-l}\int_{0}^{x}\!\xi^l \sum_{j=0}^{r-2}{{r-2}\choose
    {j}}y^{r-2-j} \xi^j d\xi\nonumber\\
    &=&\frac{1}{(m-2)!} \sum_{l=0}^{m-2}(-1)^l{{m-2}\choose{l}}x^{m-2-l}\int_{0}^{x}\!\xi^l (\xi+y)^{r-2}
    d\xi\nonumber\\
    &=&\frac{1}{(m-2)!} \int_{0}^{x}\!(x-\xi)^{m-2}(\xi+y)^{r-2}
    d\xi.\label{partOfPst1forAVG_4}
\end{eqnarray}
From (\ref{Pst1_3})--(\ref{partOfPst1forAVG_4}),
\begin{eqnarray}
    P_1(s,t,m,n)
    &=&\frac{2\, n!}{(n-m+1)!(m-2)!}\sum_{r=2}^{n-m+1}\Upsilon(m,n,r,z)\int_{0}^{x}\!(x-\xi)^{m-2}(\xi+y)^{r-2}
    d\xi\nonumber\\
    &=&2n{{n-1}\choose{m-2}}\int_{0}^{x}\ \sum_{r=2}^{n-m+1}\Upsilon(m,n,r,z)(x-\xi)^{m-2}(\xi+y)^{r-2}
    d\xi\nonumber\\
    &=&2n{{n-1}\choose{m-2}}\int_{0}^{x}\!\frac{(x-\xi)^{m-2}}{(\xi+y)^{2}}\sum_{r=2}^{n-m+1}\Upsilon(m,n,r,z)(\xi+y)^{r}
    d\xi.\label{expressionForP1stmn}
\end{eqnarray}
Here, by (\ref{Upsilon}),
\begin{eqnarray}
\lefteqn{\sum_{r=2}^{n-m+1}\Upsilon(m,n,r,z)(\xi+y)^{r}}\nonumber\\
    &=&\sum_{r=2}^{n-m+1}\left[(n-m){{n-m+1}\choose {r}}z^{n-m+1-r}-(n-m+1){{n-m}\choose {r}}z^{n-m+1-r}\right](\xi+y)^{r}\nonumber\\
    &=&(n-m)\sum_{r=2}^{n-m+1}{{n-m+1}\choose {r}}(\xi+y)^{r}z^{n-m+1-r}-(n-m+1)\sum_{r=2}^{n-m+1}{{n-m}\choose {r}}(\xi+y)^{r}z^{n-m+1-r}\nonumber\\
    &=&(n-m)[(\xi+y+z)^{n-m+1}-z^{n-m+1}-(n-m+1)(\xi+y)z^{n-m}]\nonumber\\
    & &-(n-m+1)z[(\xi+y+z)^{n-m}-z^{n-m}-(n-m)(\xi+y)z^{n-m-1}]\nonumber\\
    &=&(n-m)(\xi+y+z)^{n-m+1}-(n-m+1)z(\xi+y+z)^{n-m}+z^{n-m+1}.\label{sumOfUpsilon}
\end{eqnarray}
Substitution of (\ref{sumOfUpsilon}) into
(\ref{expressionForP1stmn}) gives the desired (\ref{Pst1_4}).
\end{proof}

\begin{proof}[Proof of Lemma \ref{lemmaForMuAvg1}]
From Lemma \ref{lemmaForP1stmn}, we have
\begin{eqnarray}
\lefteqn{\frac{1}{n}\sum_{m=2}^{n}P_1(s,t,m,n)}\nonumber\\
    &=& 2\int_{0}^{x}\!\frac{1}{(\xi+y)^2}\sum_{m=2}^{n}[\Upsilon_1(m,n,\xi,x,y,z)-\Upsilon_2(m,n,\xi,x,y,z)+\Upsilon_3(m,n,\xi,x,y,z)]\, d\xi.\nonumber\\
    & &\label{partOfMu1ForAvg}
\end{eqnarray}
Here
\begin{eqnarray}
\lefteqn{\sum_{m=2}^{n}\Upsilon_1(m,n,\xi,x,y,z)
    = (\xi+y+z)^{2}\left.\frac{d}{dw}\left[\sum_{m=2}^{n}{{n-1}\choose{m-2}}(x-\xi)^{m-2}w^{n-m}\right]\right|_{w=\xi+y+z}}\nonumber\\
    &=& (\xi+y+z)^{2}\left.\frac{d}{dw}\left\{w^{-1}[(x-\xi+w)^{n-1}-(x-\xi)^{n-1}]\right\}\right|_{w=\xi+y+z}\nonumber\\
    &=& (\xi+y+z)^{2}\left.\left\{-w^{-2}[(x-\xi+w)^{n-1}-(x-\xi)^{n-1}]+w^{-1}(n-1)(x-\xi+w)^{n-2}\right\}\right|_{w=\xi+y+z}\nonumber\\
    &=& (x-\xi)^{n-1}-1+(n-1)(\xi+y+z)
\end{eqnarray}
(note that $x+y+z = 1$). Similarly,
\begin{eqnarray}
\lefteqn{\sum_{m=2}^{n}\Upsilon_2(m,n,\xi,x,y,z)
    = z\left.\frac{d}{dw}\left[\sum_{m=2}^{n}{{n-1}\choose{m-2}}(x-\xi)^{m-2}w^{n-m+1}\right]\right|_{w=\xi+y+z}}\nonumber\\
    &=& z\left.\frac{d}{dw}\left[(x-\xi+w)^{n-1}-(x-\xi)^{n-1}\right]\right|_{w=\xi+y+z}
    = z\left.\left[(n-1)(x-\xi+w)^{n-2}\right]\right|_{w=\xi+y+z}\nonumber\\
    &=& z(n-1),
\end{eqnarray}
and
\begin{eqnarray}
\sum_{m=2}^{n}\Upsilon_3(m,n,\xi,x,y,z)
    = \sum_{m=2}^{n}{{n-1}\choose{m-2}}(x-\xi)^{m-2}z^{n-m+1}
    = (x-\xi+z)^{n-1}-(x-\xi)^{n-1}.\nonumber
\end{eqnarray}
Hence
\begin{eqnarray}
\lefteqn{\sum_{m=2}^{n}[\Upsilon_1(m,n,\xi,x,y,z)-\Upsilon_2(m,n,\xi,x,y,z)+\Upsilon_3(m,n,\xi,x,y,z)]}\nonumber\\
    &&\ \ \ \ \ \ \ \ \ \ \ \ \ \ \ \ \ \ \ \ \ \ \ \ =(n-1)(\xi+y)-1+(x-\xi+z)^{n-1}.\label{intermediateMu1Avg}
\end{eqnarray}
Therefore, from (\ref{partOfMu1ForAvg}) and
(\ref{intermediateMu1Avg}), we obtain
\begin{eqnarray}
\lefteqn{\frac{1}{n} \sum_{m=2}^{n} P_1(s,t,m,n) = 2\int_{0}^{x}\!\frac{1}{(\xi+y)^2}[(n-1)(\xi+y)-1+(x-\xi+z)^{n-1}]\,d\xi}\nonumber\\
    &=&2\int_{0}^{x}\!\frac{1}{(\xi+y)^2}\{(n-1)(\xi+y)-1+[1-(\xi+y)]^{n-1}\}\,d\xi\nonumber\\
    &=&2\int_{0}^{x}\!\frac{1}{(\xi+y)^2}\sum_{j=2}^{n-1}(-1)^j{{n-1}\choose{j}}(\xi+y)^{j}\,d\xi =2\sum_{j=2}^{n-1}(-1)^j{{n-1}\choose{j}}\int_{0}^{x}\!(\xi+y)^{j-2}\,d\xi\nonumber\\
    &=&2\sum_{j=2}^{n-1}(-1)^j{{n-1}\choose{j}}\frac{(x+y)^{j-1}-y^{j-1}}{j-1} =2\sum_{j=2}^{n-1}(-1)^j{{n-1}\choose{j}}\frac{t^{j-1}-(t-s)^{j-1}}{j-1}.\label{eqnP1Avg}
\end{eqnarray}
We complete the proof by using (\ref{eqnP1Avg}) to compute
$\int_{0}^{1}\! \int_{s}^{1}\!
\beta(s,t)\,\frac{1}{n}\sum_{m=2}^{n}P_1(s,t,m,n)\,ds\,dt$. We have
\begin{eqnarray}
\lefteqn{\int_{0}^{1}\! \int_{s}^{1}\!
\beta(s,t)\,\frac{1}{n}\sum_{m=2}^{n}P_1(s,t,m,n)\,ds\,dt}\nonumber\\
    &=& 2\int_{0}^{1}\! \int_{s}^{1} \! \beta (s,t)\ \sum_{j=2}^{n-1} (-1)^j {{n-1}\choose{j}} \frac{t^{j-1}-(t-s)^{j-1}}{j-1}\,dt\,ds\nonumber\\
    &=& 2\int_{0}^{1} \int_{s}^{1} \! \beta (s,t)\ \sum_{j=2}^{n-1} (-1)^j {{n-1}\choose{j}} \frac{t^{j-1}}{j-1}\,dt\,ds\nonumber\\
    & &-2\int_{0}^{1} \int_{s}^{1} \! \beta (s,t)\ \sum_{j=2}^{n-1} (-1)^j {{n-1}\choose{j}} \frac{(t-s)^{j-1}}{j-1}\,dt\,ds.\label{eqnMu1Avg_1}
\end{eqnarray}
Closely following the derivations shown in
(\ref{intermediateMu1})--(\ref{eqnMu1}), one can show that
\begin{eqnarray}
\lefteqn{\int_{0}^{1}\! \int_{s}^{1} \! \beta (s,t)\
    \sum_{j=2}^{n-1} (-1)^j
    {{n-1}\choose{j}} \frac{t^{j-1}}{j-1}\,dt\,ds}\nonumber\\
    &=&\sum_{j=2}^{n-1} \frac{(-1)^j  {{n-1}\choose{j}}}{j(j-1)}
    +\frac{1}{9}\sum_{j=2}^{n-1} \frac{(-1)^j  {{n-1}\choose{j}}}{j-1}
    -\sum_{j=3}^{n-1}B_j\frac{n-j+1-{{n-1}\choose{j-1}}}{j(j-1)(j-2)(1-2^{-j})}.\label{partOfMu1ForAvg_2}
\end{eqnarray}
Thus, in order to complete the proof, it remains to show that
\begin{eqnarray}
    \int_{0}^{1}\! \int_{s}^{1} \! \beta (s,t)\ \sum_{j=2}^{n-1} (-1)^j {{n-1}\choose{j}}\frac{(t-s)^
    {j-1}}{{j-1}}\,dt\,ds
    =\sum_{j=2}^{n-1}\frac{(-1)^{j}{{n-1}\choose{j}}}{(j+1)j(j-1)(1-2^{-j})}.\label{partOfMu1ForAvg_3}
\end{eqnarray}
Indeed, we have
\begin{eqnarray}
\lefteqn{\int_{0}^{1}\! \int_{s}^{1} \! \beta (s,t)\
\sum_{j=2}^{n-1} (-1)^j {{n-1}\choose{j}}\frac{(t-s)^
    {j-1}}{{j-1}}\,dt\,ds}\nonumber\\
    &=&\sum_{k=0}^{\infty}(k+1)2^{k}\int_{0}^{2^{-(k+1)}}\!\int_{2^{-(k+1)}}^{2^{-k}}\sum_{j=2}^{n-1} (-1)^j {{n-1}\choose{j}}\frac{(t-s)^
    {j-1}}{{j-1}}\,dt\,ds\nonumber\\
    &=&\sum_{k=0}^{\infty}(k+1)2^{k}\int_{0}^{2^{-(k+1)}}\!\int_{2^{-(k+1)}-s}^{2^{-k}-s}\sum_{j=2}^{n-1}
    (-1)^j {{n-1}\choose{j}}\frac{v^{j-1}}{{j-1}}\,dv\,ds\nonumber\\
    &=&\sum_{k=0}^{\infty}(k+1)2^{k}\sum_{j=2}^{n-1} (-1)^j {{n-1}\choose{j}}\int_{0}^{2^{-k}}\!\frac{v^
    {j-1}}{{j-1}}\int_{[2^{-(k+1)}-v]\bigvee 0}^{(2^{-k}-v)\bigwedge
    2^{-(k+1)}}\!ds\,dv.\label{derivationSimilarToQuicksortCase1}
\end{eqnarray}
Here
\begin{eqnarray}\label{partOfDerivationSimilarToQuicksortCase1}
\int_{[2^{-(k+1)}-v]\bigvee 0}^{(2^{-k}-v)\bigwedge
    2^{-(k+1)}}\!ds
 =   \left \{ \begin{array}{cc}
                \displaystyle 
                v \ \ \ & \mbox{if $0 \leq v \leq 2^{-(k+1)}$} \\
                \displaystyle 
                2^{-k}-v \ \ \ & \mbox{if $2^{-(k+1)} < v \leq 2^{-k}$.} \\
                \end{array}
        \right.
\end{eqnarray}
Thus
\begin{eqnarray}
\int_{0}^{2^{-k}}\!\frac{v^
    {j-1}}{{j-1}}\int_{[2^{-(k+1)}-v]\bigvee 0}^{(2^{-k}-v)\bigwedge
    2^{-(k+1)}}\!ds\,dv &=& \frac{1}{j-1} \left[\int_{0}^{2^{-(k+1)}} v^{j}\,dv
    + \int_{2^{-(k+1)}}^{2^{-k}}
    v^{j-1}(2^{-k}-v)\,dv\right]\nonumber\\
    &=& \frac{2^{-k(j+1)}(1-2^{-j})}{(j+1)j(j-1)}.\label{partOfDerivationSimilarToQuicksortCase2}
\end{eqnarray}
From (\ref{derivationSimilarToQuicksortCase1}) and
(\ref{partOfDerivationSimilarToQuicksortCase2}), we obtain
\begin{eqnarray}
\lefteqn{\int_{0}^{1}\! \int_{s}^{1} \! \beta (s,t)\
\sum_{j=2}^{n-1} (-1)^j {{n-1}\choose{j}}\frac{(t-s)^
    {j-1}}{{j-1}}\,dt\,ds}\nonumber\\
    &=&\sum_{k=0}^{\infty}(k+1)2^{k}\sum_{j=2}^{n-1} (-1)^j
    {{n-1}\choose{j}}\frac{2^{-k(j+1)}(1-2^{-j})}{(j+1)j(j-1)}\nonumber\\
    &=&\sum_{j=2}^{n-1} (-1)^j
    {{n-1}\choose{j}}\frac{1-2^{-j}}{(j+1)j(j-1)}\sum_{k=0}^{\infty}(k+1)2^{-kj}\nonumber\\
    &=&\sum_{j=2}^{n-1} (-1)^j
    {{n-1}\choose{j}}\frac{1-2^{-j}}{(j+1)j(j-1)}\frac{1}{(1-2^{-j})^2}\nonumber\\
    &=&\sum_{j=2}^{n-1}\frac{(-1)^{j}{{n-1}\choose{j}}}{(j+1)j(j-1)(1-2^{-j})},\label{derivationSimilarToQuicksortCase2}
\end{eqnarray}
and (\ref{partOfMu1ForAvg_3}) is proved.
\end{proof}


\subsubsection{Exact Computation of $\mu_2(\bar{m},n)$ and $\mu(\bar{m},n)$}\label{derivationOfMuAvg2AndMuAvg}
The derivations for obtaining a computationally preferable exact
expression for $\mu_2(\bar{m},n)$ are entirely analogous to those
for $\mu_1(\bar{m},n)$ described in the previous section (Section
\ref{derivationOfMuAvg1}).  Thus we omit details.  As described in
Section \ref{derivationOfMu1}, $P_2(s,t,m,n)$ is zero for $m=1$ and
for $m=n$, so, from (\ref{Mu_1_2_3}),
\begin{eqnarray}
    \mu_2(\bar{m},n) &=& \int_{0}^{1} \int_{s}^{1} \! \beta (s,t)\ \frac{1}{n} \sum_{m=2}^{n-1} P_2(s,t,m,n)\ dt\,
    ds.\label{Mu2Avg_1}
\end{eqnarray}
Therefore we first derive a computationally desirable expression for
$\frac{1}{n}\sum_{m=2}^{n-1}P_2(s,t,m,n)$.  Again, let $x:=s,\
y:=t-s,\ z:=1-t.$ Then
\begin{eqnarray}
\lefteqn{\frac{1}{n}\sum_{m=2}^{n-1}P_2(s,t,m,n)}\nonumber\\
    &=& \frac{1}{n}\sum_{m=2}^{n-1}\sum_{1 \leq i \leq m < j \leq n} \frac{2}{j-i+1} {n\choose{i-1,1,j-i-1,1,n-j}}\ x^{i-1} y^{j-i-1}
    z^{n-j}\nonumber\\
    &=& \frac{1}{n}\sum_{m=2}^{n-1}S_1(m,n,x,y,z) - \frac{1}{n}\sum_{m=2}^{n-1}S_2(m,n,x,y,z)-\frac{1}{n}\sum_{m=2}^{n-1}S_3(m,n,x,y,z), \label{sumPst2_1}
\end{eqnarray}
where
\begin{eqnarray}
    S_1(m,n,x,y,z) &:=& \sum_{1 \leq i < j \leq n} \frac{2}{j-i+1} {n\choose{i-1,1,j-i-1,1,n-j}}\ x^{i-1} y^{j-i-1}
    z^{n-j},\nonumber\\
    S_2(m,n,x,y,z) &:=& \sum_{m \leq i < j \leq n} \frac{2}{j-i+1} {n\choose{i-1,1,j-i-1,1,n-j}}\ x^{i-1} y^{j-i-1}
    z^{n-j},\nonumber\\
    S_3(m,n,x,y,z) &:=& \sum_{1 \leq i < j \leq m} \frac{2}{j-i+1} {n\choose{i-1,1,j-i-1,1,n-j}}\ x^{i-1} y^{j-i-1}
    z^{n-j}.\nonumber
\end{eqnarray}
Fill and Janson \cite{fj2004} showed that $S_1(m,n,x,y,z) =
2\sum_{j=2}^{n}(-1)^j{{n}\choose{j}}(t-s)^{j-2}$.  Hence
\begin{eqnarray}
    \frac{1}{n}\sum_{m=2}^{n-1}S_1(m,n,x,y,z) = \frac{2(n-2)}{n}\sum_{j=2}^{n}(-1)^j{{n}\choose{j}}(t-s)^{j-2}.\label{sumS1}
\end{eqnarray}
Following the derivations shown in (\ref{Pst1_1}) through
(\ref{eqnP1Avg}), one can show that
\begin{eqnarray}
\frac{1}{n}\sum_{m=2}^{n-1}S_2(m,n,x,y,z) &=& 2y^{-2}x[(x+z)^{n-1}-1+y(n-1)]\label{usefulExpressionForS3}\\
    &=&2(t-s)^{-2}s\{[1-(t-s)]^{n-1}-1+(t-s)(n-1)\}\nonumber\\
    &=&2s\sum_{j=2}^{n-1}(-1)^{j}{{n-1}\choose{j}}(t-s)^{j-2}.\label{sumS2}
\end{eqnarray}
To obtain a similar expression for
$\frac{1}{n}\sum_{m=2}^{n-1}S_3(m,n,x,y,z)$, we note that, letting
$m':=n+1-m,\ i':=n+1-j,\ j':=n+1-i$,
\begin{eqnarray*}
    S_3(m,n,x,y,z)
    &=&\sum_{m' \leq i' < j' \leq n} \frac{2}{j'-i'+1} {n\choose{n-j',1,j'-i'-1,1,i'-1}}\ x^{n-j'} y^{j'-i'-1}
    z^{i'-1}\\
    &=&S_2(n+1-m,n,z,y,x).
\end{eqnarray*}
Thus
\begin{eqnarray}
    \frac{1}{n}\sum_{m=2}^{n-1}S_3(m,n,x,y,z)&=&\frac{1}{n}\sum_{m=2}^{n-1}S_2(n+1-m,n,z,y,x)\nonumber\\
    &=&\frac{1}{n}\sum_{m=2}^{n-1}S_2(m,n,z,y,x).\label{S1andS2}
\end{eqnarray}
Inspecting (\ref{usefulExpressionForS3})--(\ref{S1andS2}), we find
\begin{eqnarray}
    \frac{1}{n}\sum_{m=2}^{n-1}S_3(m,n,x,y,z)
    =2(1-t)\sum_{j=2}^{n-1}(-1)^{j}{{n-1}\choose{j}}(t-s)^{j-2}.\label{sumS3}
\end{eqnarray}
From (\ref{sumPst2_1}), (\ref{sumS1}), (\ref{sumS2}), and
(\ref{sumS3}),
\begin{eqnarray}
    \frac{1}{n} \sum_{m=1}^{n-1} P_2(s,t,m,n)
    &=&\frac{2(n-2)}{n} \sum_{j=2}^{n} (-1)^j {{n}\choose{j}} (t-s)^{j-2}
     - 2s\sum_{j=2}^{n-1} (-1)^j {{n-1}\choose{j}} (t-s)^{j-2}\nonumber\\
    & &- 2(1-t)\sum_{j=2}^{n-1} (-1)^j {{n-1}\choose{j}} (t-s)^{j-2}\nonumber\\
    &=& \frac{2(n-2)}{n} \sum_{j=2}^{n} (-1)^j {{n}\choose{j}} (t-s)^{j-2}
         -2\sum_{j=2}^{n-1} (-1)^j {{n-1}\choose{j}} (t-s)^{j-2}\nonumber\\
    & &+2\sum_{j=2}^{n-1} (-1)^j {{n-1}\choose{j}} (t-s)^{j-1}.\label{eqnP2Avg}
\end{eqnarray}
Hence, from (\ref{Mu2Avg_1}) and (\ref{eqnP2Avg}),
\begin{eqnarray}
    \mu_2(\bar{m},n)&=&\frac{2(n-2)}{n} \int_{0}^{1}\!\int_{s}^{1}\!\beta(s,t)\sum_{j=2}^{n} (-1)^j {{n}\choose{j}} (t-s)^{j-2}\,dt\,ds\nonumber \\
    &  & - 2\int_{0}^{1}\!\int_{s}^{1} \beta(s,t) \sum_{j=2}^{n-1} (-1)^j {{n-1}\choose{j}}
    (t-s)^{j-2}\,dt\,ds\nonumber\\
    &  & + 2\int_{0}^{1}\!\int_{s}^{1}\!\beta(s,t)\sum_{j=2}^{n-1} (-1)^j {{n-1}\choose{j}}
    (t-s)^{j-1}\,dt\,ds.\label{Mu2Avg_2}
\end{eqnarray}
Fill and Janson \cite{fj2004} showed that
\begin{eqnarray}
    \int_{0}^{1}\!\int_{s}^{1}\beta(s,t)\sum_{j=2}^{n} (-1)^j {{n}\choose{j}} (t-s)^{j-2}\,dt\,ds
    =\sum_{j=2}^{n}\frac{(-1)^{j}{{n}\choose{j}}}{j(j-1)[1-2^{-(j-1)}]}.\label{partOfMu2Avg_1}
\end{eqnarray}
A careful term-by-term inspection of the derivations shown in
(\ref{derivationSimilarToQuicksortCase1})--(\ref{derivationSimilarToQuicksortCase2})
reveals that
\begin{eqnarray}
    \int_{0}^{1}\!\int_{s}^{1}\!\beta(s,t)\sum_{j=2}^{n-1} (-1)^j {{n-1}\choose{j}}
    (t-s)^{j-2}\,dt\,ds &=&
    \sum_{j=2}^{n-1}\frac{(-1)^{j}{{n-1}\choose{j}}}{j(j-1)[1-2^{-(j-1)}]},\label{partOfMu2Avg_2}\\
    \int_{0}^{1}\!\int_{s}^{1}\!\beta(s,t)\sum_{j=2}^{n-1} (-1)^j {{n-1}\choose{j}}
    (t-s)^{j-1}\,dt\,ds &=&
    \sum_{j=2}^{n-1}\frac{(-1)^{j}{{n-1}\choose{j}}}{j(j+1)(1-2^{-j})}.\label{partOfMu2Avg_3}
\end{eqnarray}
Combining (\ref{Mu2Avg_2})--(\ref{partOfMu2Avg_3}), we obtain
\begin{eqnarray}
    \mu_2(\bar{m},n) &=&\frac{2(n-2)}{n}\sum_{j=0}^{n-2}\frac{(-1)^{j}{{n}\choose{j+2}}}{(j+1)(j+2)[1-2^{-(j+1)}]}
    -2\sum_{j=2}^{n}\frac{(-1)^{j}{{n}\choose{j}}}{j(j-1)[1-2^{-(j-1)}]}+2(n-1)\nonumber\\
    &=&-\frac{4}{n}\sum_{j=2}^{n}\frac{(-1)^{j}{{n}\choose{j}}}{j(j-1)[1-2^{-(j-1)}]}+2(n-1).\label{Mu2Avg_3}
\end{eqnarray}

Finally, we complete the exact computation of $\mu(\bar{m},n)$. From
(\ref{eqnMuAvg}), (\ref{eqnMu1Avg_2}), and (\ref{Mu2Avg_3}), we have
\begin{eqnarray}
\lefteqn{\mu (\bar{m},n) =
    2\mu_1(\bar{m},n)+\mu_2(\bar{m},n)}\nonumber\\
    &=& 2(n-1)-8\sum_{j=3}^{n} \frac{(-1)^j
    {{n-1}\choose{j-1}}}{j(j-1)(j-2)} + \frac{4}{n}
    \sum_{j=2}^{n-1}B_j\frac{n-j+1-{{n}\choose{j}}}{j(j-1)(1-2^{-j})}\nonumber\\
    & &
    +\frac{4}{9}\sum_{j=2}^{n-1} \frac{(-1)^j  {{n-1}\choose{j}}}{j-1}
    -4\sum_{j=3}^{n-1}B_j\frac{n-j+1-{{n-1}\choose{j-1}}}{j(j-1)(j-2)(1-2^{-j})}
    -4\sum_{j=2}^{n-1}\frac{(-1)^{j}{{n-1}\choose{j}}}{(j+1)j(j-1)(1-2^{-j})}\nonumber\\
    &
    &-\frac{4}{n}\sum_{j=2}^{n}\frac{(-1)^{j}{{n}\choose{j}}}{j(j-1)[1-2^{-(j-1)}]}+2(n-1).\label{intermediateExactMuAvg}
\end{eqnarray}
We rewrite or combine some of the terms in
(\ref{intermediateExactMuAvg}) for the asymptotic analysis of
$\mu(\bar{m},n)$ described in the next section. We define
\begin{eqnarray}
    F_1(n) &:=& \sum_{j=3}^{n}
    \frac{(-1)^j{{n}\choose{j}}}{(j-1)(j-2)},\nonumber\\
    F_2(n) &:=& \sum_{j=2}^{n-1}\frac{B_j}{j(1-2^{-j})}\left[\frac{n-{{n}\choose{j}}}{j-1}-1\right],\nonumber\\
    F_3(n) &:=& \sum_{j=2}^{n-1}
    \frac{(-1)^j{{n-1}\choose{j}}}{j-1},\nonumber\\
    F_4(n) &:=& \sum_{j=3}^{n-1}\frac{B_j}{j(j-1)(1-2^{-j})}\left[\frac{n-1-{{n-1}\choose{j-1}}}{j-2}-1\right],\nonumber\\
    F_5(n) &:=& \sum_{j=3}^{n} \frac{(-1)^j{{n}\choose{j}}}{j(j-1)(j-2)[1 -
    2^{-(j-1)}]}.\nonumber
\end{eqnarray}
The second, third, fourth, and fifth terms in
(\ref{intermediateExactMuAvg}) can be written as
$-\frac{8}{n}F_1(n)$, $\frac{4}{n} F_2(n)$, $\frac{4}{9} F_3(n)$,
and $-4 F_4(n)$, respectively. The last three terms in
(\ref{intermediateExactMuAvg}) can be combined as follows:
\begin{eqnarray}
\lefteqn{-4\sum_{j=2}^{n-1}\frac{(-1)^{j}{{n-1}\choose{j}}}{(j+1)j(j-1)(1-2^{-j})}
    -\frac{4}{n}\sum_{j=2}^{n}\frac{(-1)^{j}{{n}\choose{j}}}{j(j-1)[1-2^{-(j-1)}]}+2(n-1)}\nonumber\\
    & &\ \ \ \ \ = \frac{4}{n}\sum_{j=3}^{n}\frac{(-1)^{j}{{n}\choose{j}}}{(j-1)(j-2)[1-2^{-(j-1)}]}
    -\frac{4}{n}\sum_{j=2}^{n}\frac{(-1)^{j}{{n}\choose{j}}}{j(j-1)[1-2^{-(j-1)}]}+2(n-1)\nonumber\\
    & &\ \ \ \ \ =\frac{8}{n}\sum_{j=3}^{n}\frac{(-1)^{j}{{n}\choose{j}}}{j(j-1)(j-2)[1-2^{-(j-1)}]}
    = \frac{8}{n} F_5(n).\nonumber
\end{eqnarray}
Therefore
\begin{eqnarray}
    \mu(\bar{m},n) = 2(n-1) - \mbox{$\frac{8}{n}$} F_1(n) + \mbox{$\frac{4}{n}$} F_2(n) + \mbox{$\frac{4}{9}$} F_3(n) - 4 F_4(n)
    + \mbox{$\frac{8}{n}$}F_5(n).\label{eqnMuAvgFinal}
\end{eqnarray}


\subsection{Asymptotic Analysis of $\mu(\bar{m},n)$}

We derive an asymptotic expression for $\mu(\bar{m},n)$ shown in
(\ref{eqnMuAvgFinal}). The computations described in this section
are analogous to those in Section \ref{asymptoticAnalysisOfMu1}.
Hence we merely sketch details to derive the asymptotic expression.
First, we analyze $F_1(n)$. A routine complex-analytical argument
similar to (but much easier than) the one described in Section
\ref{asymptoticAnalysisOfMu1} shows that
\begin{eqnarray}
  F_1(n) &=&
  (-1)^{n+1}\sum_{k=0}^{2}\textrm{Res}_{s=k}\left[\frac{n!}{s(s-1)^2(s-2)^2(s-3)\cdots(s-n)}\right]\nonumber\\
  &=&(-1)^{n+1}\left[\frac{(-1)^n}{2} + (-1)^n n H_{n-1} +
  \frac{(-1)^n}{2}n(n-1)\left(H_{n-2}-\frac{5}{2}\right)\right]\nonumber\\
  &=&
  -\frac{1}{2}n(n-1)H_{n-2}+\frac{5}{4}n(n-1)-nH_{n-1} -
  \frac{1}{2}\nonumber\\
  &=&-\frac{1}{2}n^2 \ln n +
  \left(\frac{5}{4}-\frac{\gamma}{2}\right)n^2 - n \ln n +
  \frac{n^2}{2(n-1)}-(\gamma+1)n+O(1).\label{F1}
\end{eqnarray}

Since $F_2(n)$ is equal to $t_n$, which is defined at (\ref{t_n1})
and analyzed in Section \ref{asymptoticAnalysisOfMu1}, we already
have an asymptotic expression for $F_2(n)$. Next we derive an
asymptotic expression for $F_3(n)$:
\begin{eqnarray}
F_3(n)&=&(-1)^n \sum_{k=0}^{1} \textrm{Res}_{s=k} \left\{\frac{(n-1)!}{s(s-1)^2(s-2)\cdots[s-(n-1)]}\right\}\nonumber\\
&=&nH_{n-2}-n-H_{n-2}+2\nonumber\\
&=&n\ln n + (\gamma-1)n-\ln n + O(1).\label{F3}
\end{eqnarray}

To obtain an asymptotic expression for $F_4(n)$, we closely follow
the approach of Section \ref{asymptoticAnalysisOfMu1}. Let
$\tilde{u}_n := F_4(n+1) - F_4(n)$. Then
\begin{eqnarray*}
\tilde{u}_n=-\sum_{j=3}^{n}\frac{B_j}{j(j-1)(j-2)(1-2^{-j})}\left[{{n-1}\choose{j-2}}-1\right].
\end{eqnarray*}
Let $\tilde{v}_n:=\tilde{u}_{n+1}-\tilde{u}_n$.  Then, by
computations similar to those performed for $v_n$ in Section
\ref{asymptoticAnalysisOfMu1},
\begin{eqnarray*}
\tilde{v}_n
&=&-\sum_{k=0}^{n-2}(-1)^k\frac{\zeta(-2-k)}{(k+2)(k+1)[1-2^{-(k+3)}]}{{n-1}\choose{k}}\nonumber\\
&=&(-1)^{n+1}\sum_{k=1}^{3}\textrm{Res}_{s=-k}\left\{
\frac{\zeta(-2-s)}{(s+2)(s+1)[1-2^{-(s+3)}]}\frac{(n-1)!}{s(s-1)\cdots[s-(n-1)]}\right\}\nonumber\\
& &+(-1)^{n+1}\sum_{k\in
\mathbb{Z}\backslash\{0\}}\textrm{Res}_{s=-3+\chi_k}\left\{
\frac{\zeta(-2-s)}{(s+2)(s+1)[1-2^{-(s+3)}]}\frac{(n-1)!}{s(s-1)\cdots[s-(n-1)]}\right\}\nonumber\\
&=&\frac{1}{9n}-\frac{1}{n(n+1)}
-\frac{1}{n(n+1)(n+2)}\left[\frac{\gamma}{\ln
2}-\frac{1}{2}-\frac{H_{n+2}}{n+2}\right] - \xi_n,
\end{eqnarray*}
where
\begin{eqnarray*}
\xi_n := \sum_{k \in
\mathbb{Z}\backslash\{0\}}\frac{\zeta(1-\chi_k)\Gamma(1-\chi_k)\Gamma(n)}{(\ln
2)\Gamma(n+3-\chi_k)}.
\end{eqnarray*}
Hence
\begin{eqnarray*}
\tilde{u}_n &=& \tilde{u}_2+\sum_{j=2}^{n-1}\tilde{v}_n\nonumber\\
&=&\frac{1}{9}H_{n-1}+\tilde{a}+\tilde{\xi}_n-\frac{1}{2\ln
2}\left(\frac{H_n}{n}-\frac{H_{n+1}}{n+1}\right)+\frac{1}{n}-\frac{3+\ln
2-2\gamma}{4\ln2}\frac{1}{n(n+1)},\nonumber\\
\end{eqnarray*}
where
\begin{eqnarray*}
\tilde{a}&:=&\frac{7}{36\ln 2}-\frac{41}{72}-\frac{\gamma}{12\ln
2}-\sum_{k \in
Z\backslash\{0\}}\frac{\zeta(1-\chi_k)\Gamma(1-\chi_k)}{(\ln
2)(2-\chi_k)\Gamma(4-\chi_k)},\nonumber\\
\tilde{\xi}_n&:=&\sum_{k \in
Z\backslash\{0\}}\frac{\zeta(1-\chi_k)\Gamma(1-\chi_k)\Gamma(n)}{(\ln
2)(2-\chi_k)\Gamma(n+2-\chi_k)}.\nonumber
\end{eqnarray*}
Thus
\begin{eqnarray}
F_4(n) &=& F_4(2) + \sum_{j=2}^{n-1}\tilde{u}_j\nonumber\\
&=&\frac{1}{9}nH_{n-1}+\frac{8}{9}H_{n-1}+ \left(
\tilde{a}-\frac{1}{9} \right) n-\frac{8}{9}-\frac{3}{8\ln
2}-\frac{3+\ln 2
-2\gamma}{8\ln2}-2 \tilde{a} + \tilde{b} -\tilde{\tilde{\xi}}_n\nonumber\\
& &
+\frac{1}{2\ln2}\frac{H_n}{n}+\frac{3+\ln2-2\gamma}{4\ln2}\frac{1}{n},
\end{eqnarray}
where
\begin{eqnarray}
\tilde{b} &:=&\sum_{k \in
Z\backslash\{0\}}\frac{\zeta(1-\chi_k)\Gamma(1-\chi_k)}{(\ln
2)(2-\chi_k)(1-\chi_k)\Gamma(3-\chi_k)},\nonumber\\
\tilde{\tilde{\xi}}_n&:=&\sum_{k \in
Z\backslash\{0\}}\frac{\zeta(1-\chi_k)\Gamma(1-\chi_k)\Gamma(n)}{(\ln
2)(2-\chi_k)(1-\chi_k)\Gamma(n+1-\chi_k)}.\nonumber
\end{eqnarray}
Therefore
\begin{eqnarray}
F_4(n)=\frac{1}{9}n\ln n + \left( \tilde{a}
+\frac{1}{9}\gamma-\frac{1}{9} \right) n+\frac{8}{9}\ln n +
O(1).\label{F4}
\end{eqnarray}

Finally, we analyze $F_5(n)$. By computations that are entirely
analogous to those performed for $F_1(n),\ F_2(n)$, and $F_4(n)$,
\begin{eqnarray}
F_5(n)&=&(-1)^{n+1}\sum_{k=0}^{2}\textrm{Res}_{s=k}\left\{
\frac{n!}{[1-2^{-(s-1)}]s^2(s-1)^2(s-2)^2(s-3)\cdots(s-n)}\right\}\nonumber\\
& &+(-1)^{n+1}\sum_{k\in
Z\backslash\{0\}}\textrm{Res}_{s=1+\chi_k}\left\{
\frac{n!}{[1-2^{-(s-1)}]s^2(s-1)^2(s-2)^2(s-3)\cdots(s-n)}\right\}\nonumber\\
&=&\frac{1}{4}(2H_n+3+4\ln2)-\frac{n(n-1)}{2}(H_{n-2}-\ln2-3)\nonumber\\
& &-
n\left[\frac{1}{2\ln2}(H_{n-1})^2+\left(\frac{1}{2}-\frac{1}{\ln2}\right)H_{n-1}+\frac{1}{2\ln2}
H_{n-1}^{(2)}+\frac{2}{\ln2}+\frac{\ln2}{12}-\frac{1}{2}\right]\nonumber\\
& &+\sum_{k \in
\mathbb{Z}\backslash\{0\}}\frac{\Gamma(-1-\chi_k)\Gamma(n+1)}{(\ln
2)\chi_k(\chi_k^2-1)\Gamma(n-1-\chi_k)}\nonumber\\
&=&-\frac{1}{2}n^2\ln n +
\frac{3+\ln2-\gamma}{2}n^2-\frac{1}{2\ln2}n(\ln n)^2
+\left(\frac{1}{\ln2}-\frac{1}{2}\right)n\ln n + O(n).\label{F5}
\end{eqnarray}

Therefore, from (\ref{eqnMuAvgFinal})--(\ref{F3}) and
(\ref{F4})--(\ref{F5}), we obtain the following asymptotic formula
for $\mu(\bar{m},n)$:
\begin{eqnarray}
\mu(\bar{m},n)=4(1+\ln2-\tilde{a})n-\frac{4}{\ln2}(\ln n)^2 +
4\left(\frac{2}{\ln2}-1\right)\ln n +
O(1).\label{asymptoticFormulaForMuAvg}
\end{eqnarray}
The asymptotic slope $4(1+\ln2-\tilde{a})$ is approximately 8.20731.

\section{Derivation of a Closed Formula for $\mu(\lowercase{m},\lowercase{n})$}\label{closedForm}
\setlength\parindent{.5in}
\graphicspath{{closedForm/}}
The exact expression for $\mu (m,n)$
obtained in Section \ref{preliminaries} [see (\ref{eqnMu})] involves
infinite summation and integration.  Hence it is not a preferable
form for numerically computing the expectation.  In this section, we
establish another exact expression for $\mu (m,n)$ that only
involves finite summation. We also use the formula to compute $\mu
(m,n)$ for $m = 1,\ldots,n$, $n=2,\ldots,20$. \\
\indent As described in Section \ref{preliminaries}, it follows from
equations (\ref{eqnP})--(\ref{eqnMu}) that
\begin{eqnarray}
  \mu(m,n) &=& \mu_1(m,n) + \mu_2(m,n) + \mu_3(m,n),
\end{eqnarray}
where, for $q = 1,2,3,$
\begin{eqnarray}
\mu_q(m,n):=\sum_{k=0}^{\infty} \sum_{l=1}^{2^k} \int_{s =
(l-1)2^{-k}}^{(l-\frac{1}{2})2^{-k}} \int_{t =
(l-\frac{1}{2})2^{-k}}^{l2^{-k}} (k+1) P_q(s,t,m,n)\ dt\ ds.
\label{eqnMu_1}
\end{eqnarray}
The same technique can be applied to eliminate the infinite
summation and integration from each $\mu_q(m,n)$. We describe the
technique for obtaining a closed expression of $\mu_1(m,n)$ in
detail.\\
\indent First, we transform $P_1(s,t,m,n)$ shown in (\ref{eqnP1}) so
that we can eliminate the integration in $\mu_1 (m,n)$. Define
\begin{equation}
    C_1(i,j) := I\{1 \leq m \leq i < j \leq
    n\}\frac{2}{j-m+1}{{n}\choose{i-1,1,j-i-1,1,n-j}},\label{C1}
\end{equation}
where $I\{1 \leq m \leq i < j \leq n\}$ is an indicator function
that equals 1 if the event in braces holds and 0 otherwise. Since
\begin{eqnarray*}
\lefteqn{s^{i-1}(t-s)^{j-i-1}(1-t)^{n-j}}\nonumber\\
&=&s^{i-1}\sum_{u=0}^{j-i-1}{{j-i-1}\choose{u}}t^u(-1)^{j-i-1-u}s^{j-i-1-u}\sum_{v=0}^{n-j}{{n-j}\choose{v}}(-1)^{n-j-v}t^{n-j-v},
\end{eqnarray*}
it follows that
\begin{eqnarray}
    P_1(s,t,m,n) &=& \sum_{m \leq i < j \leq n} C_1(i,j) \sum_{u=0}^{j-i-1} \sum_{v = 0}^{n-j} {{j-i-1}\choose{u}}{{n-j}\choose{v}}s^{j-u-2}t^{n-j-v+u}(-1)^{n-i-u-v-1}\nonumber \\
    &=& \sum_{m \leq i < j \leq n} C_1(i,j) \sum_{f=i-1}^{j-2} \sum_{h = j-f-2}^{n-f-2} s^f t^h {{j-i-1}\choose{f-i+1}}{{n-j}\choose{h-j+f+2}}
    (-1)^{h-i-j+1}\nonumber\\
    & = & \sum_{f=m-1}^{n-2} \sum_{h = 0}^{n-f-2} s^f t^h C_2(f,h),\label{eqnP1withC2}
\end{eqnarray}
where 
\begin{equation*}
    C_2(f,h) := \sum_{i=m}^{f+1} \sum_{j = f+2}^{f+h+2} C_1(i,j){{j-i-1}\choose{f-i+1}}{{n-j}\choose{h-j+f+2}} (-1)^{h-i-j+1}.
\end{equation*}
Thus, from (\ref{eqnMu_1}) and (\ref{eqnP1withC2}), we can eliminate
the integration in $\mu_1(m,n)$ and express it using polynomials in
$l$:
\begin{eqnarray}
\lefteqn{\mu_1(m,n)}\nonumber\\
    &=&\sum_{f=m-1}^{n-2} \sum_{h = 0}^{n-f-2} C_3(f,h) \sum_{k=0}^{\infty} (k+1) \sum_{l=1}^{2^k}
    2^{-k(f+h+2)}[ l^{h+1}-( l-\mbox{$\frac{1}{2}$} )^{h+1}] [
    (l-\mbox{$\frac{1}{2}$})^{f+1}-(l-1)^{f+1}],\nonumber\\
    & &\label{eqnMu_1WithC3}
\end{eqnarray}
where
\begin{equation*}
    C_3(f,h) := \frac{1}{(n+1)(f+1)}C_2(f,h).
\end{equation*}
Note that
\begin{eqnarray*}
l^{h+1}-\left( l-\frac{1}{2} \right)^{h+1} &=&
-\sum_{j=0}^{h}{{h+1}\choose{j}}l^j\left(-\frac{1}{2}\right)^{h+1-j},\\
\left(l-\frac{1}{2} \right)^{f+1}-(l-1)^{f+1} &=&
-\sum_{j'=0}^{f}{{f+1}\choose{j'}}l^{j'}(-1)^{f+1-j'}\left[1-\left(\frac{1}{2}\right)^{f+1-j'}\right].\nonumber
\end{eqnarray*}
Hence
\begin{eqnarray}
\lefteqn{\left[l^{h+1}-\left( l-\frac{1}{2}
\right)^{h+1}\right]\left[\left(l-\frac{1}{2}
\right)^{f+1}-(l-1)^{f+1}\right]}\nonumber\\
&=&\sum_{j'=0}^{f}\sum_{j=0}^{h}{{f+1}\choose{j'}}{{h+1}\choose{j}}(-1)^{f+h-j'-j}
\left[1-\left(\frac{1}{2}\right)^{f+1-j'}\right]\left(\frac{1}{2}\right)^{h+1-j}l^{j'+j},\nonumber
\end{eqnarray}
which can be rearranged to
\begin{eqnarray}
\sum_{j=1}^{f+h+1}C_4(f,h,j)l^{j-1},\label{partOfMu1withC3}
\end{eqnarray}
 where 
\begin{eqnarray}
\lefteqn{C_4(f,h,j)}\nonumber\\
    &:=& (-1)^{f+h-j+1} \left(\frac{1}{2}\right)^{h-j+2} \sum_{j'=0\bigvee(j-1-h)}^{(j-1)\bigwedge f}
    {{f+1}\choose{j'}}{{h+1}\choose{j-1-j'}}
    \left[1-\left(\frac{1}{2}\right)^{f+1-j'}\right]\left(\frac{1}{2}\right)^{j'}.\nonumber
\end{eqnarray}
Therefore, from (\ref{eqnMu_1WithC3})--(\ref{partOfMu1withC3}), we
obtain
\begin{eqnarray}
\mu_1(m,n) &=& \sum_{f=m-1}^{n-2} \sum_{h = 0}^{n-f-2} C_3(f,h)
\sum_{k=0}^{\infty}
(k+1) \sum_{l=1}^{2^k} 2^{-k(f+h+2)}\sum_{j=1}^{f+h+1} C_4(f,h,j) l^{j-1}\nonumber\\
&=&\sum_{f=m-1}^{n-2} \sum_{h = 0}^{n-f-2} \sum_{j=1}^{f+h+1}
C_5(f,h,j) \sum_{k=0}^{\infty} (k+1) 2^{-k(f+h+2)} \sum_{l=1}^{2^k}
  l^{j-1},\nonumber
\end{eqnarray}
where
\begin{equation}\label{c5ForMu1}
    C_5(f,h,j) := C_3(f,h) \cdot C_4(f,h,j).\nonumber
\end{equation}
Here, as described in Section \ref{derivationOfMu1},
\begin{eqnarray*}
\sum_{l=1}^{2^k}l^{j-1}=\sum_{r=0}^{j-1}a_{j,r}2^{k(j-r)},
\end{eqnarray*}
where $a_{j,r}$ is defined by (\ref{a_jr}).  Now define
\begin{eqnarray*}
C_6(f,h,j,r) := a_{j,r}\ C_5(f,h,j).
\end{eqnarray*}
Then
\begin{eqnarray}
  \mu_1(m,n) &=& \sum_{f=m-1}^{n-2} \sum_{h = 0}^{n-f-2} \sum_{j=1}^{f+h+1}\sum_{r=0}^{j-1} a_{j,r}\, C_5(f,h,j) \sum_{k=0}^{\infty} (k+1)
  2^{-k(f+h+2+r-j)}\nonumber\\
  &=& \sum_{f=m-1}^{n-2} \sum_{h = 0}^{n-f-2} \sum_{j=1}^{f+h+1} \sum_{r=0}^{j-1} C_6(f,h,j,r) [1 -
  2^{-(f+h+2+r-j)}]^{-2}\nonumber\\
  &=& \sum_{a=1}^{n-1}  C_7(a) (1 - 2^{-a})^{-2},
\end{eqnarray}
where
\begin{eqnarray*}
  C_7(a) := \sum_{f=m-1}^{n-2} \sum_{h = \alpha}^{n-f-2} \sum_{j = \beta}^{f+h+1} C_6(f,h,j, a+j-(f+h+2)),
\end{eqnarray*}
in which $\alpha:=0\bigvee(a-f-1)$ and $\beta := 1\bigvee(f+h+2-a)$.\\
\indent The procedure described above can be applied to derive
analogous exact formulae for $\mu_2(m,n)$ and $\mu_3(m,n)$. In order
to derive the analogous exact formula for $\mu_2(m,n)$, one need
only start the derivation by changing the indicator function in
$C_1(i,j)$ [see (\ref{C1})] to $I\{1 \leq i < m < j \leq n\}$ and
follow each step of the procedure; for $\mu_3(m,n)$, start the
derivation by changing the indicator function to $I\{1 \leq i < j
\leq m \leq
n\}$.\\
\indent Using the closed exact formulae of $\mu_1(m,n)$,
$\mu_2(m,n)$, and $\mu_3(m,n)$, we computed $\mu(m,n)$ for $n =
2,3,\ldots,20$ and $m=1,2,\ldots,n$. Figure \ref{expectation} shows
the results, which suggest the following: (i) for fixed $n$, $\mu
(m,n)$ increases in $m$ for $m \leq \frac{n+1}{2}$ and is symmetric
about $\frac{n+1}{2}$;  (ii) for fixed $m$, $\mu (m,n)$ increases in
$n$ (asymptotically linearly).

\begin{figure}[!htbp]
\centering \vspace{1.0cm}
\includegraphics[height = 12cm]{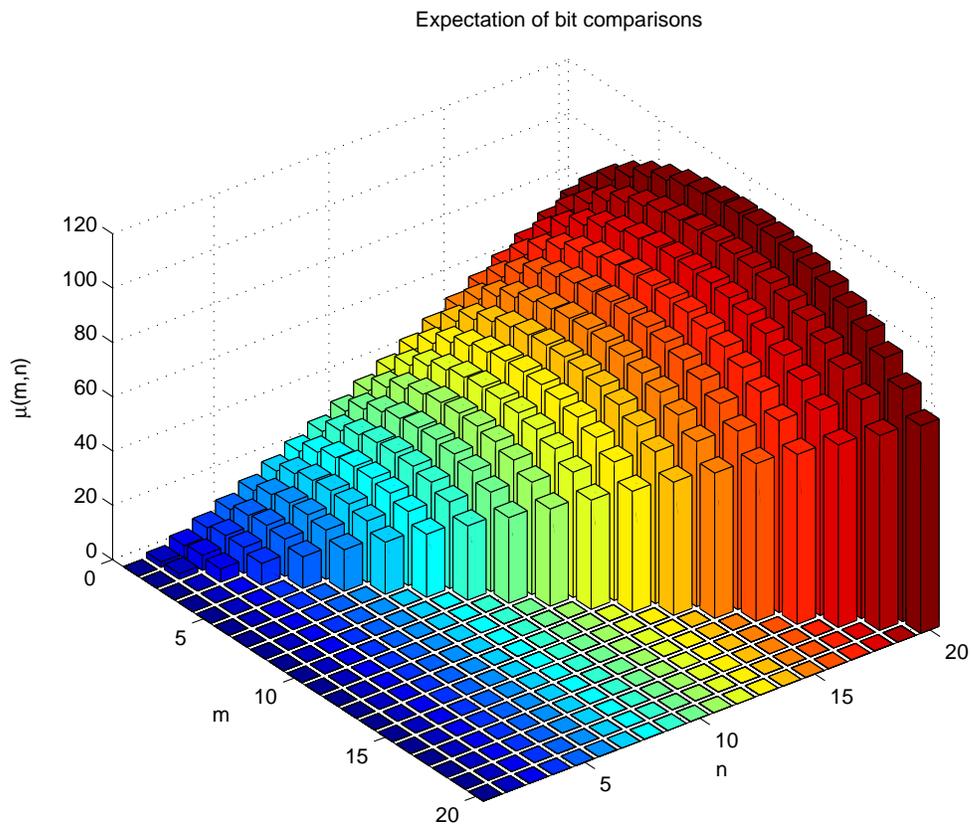}
\vspace{1.0cm} \caption[Expected number of bit comparisons for
{\tt Quickselect}] {Expected number of bit comparisons for {\tt Quickselect}.
The closed formulae for $\mu_1(m,n),\ \mu_2(m,n)$, and $\mu_3(m,n)$
were used to compute $\mu(m,n)$ for $n=1,2,\ldots,20$ ($n$
represents the number of keys) and $m=1,2,\ldots,n$ ($m$ represents
the rank of the target key). } \label{expectation}
\end{figure}

\section{Discussion}\label{discussion}
\setlength\parindent{.5in}
Our investigation of the bit complexity of {\tt Quickselect} revealed that
the expected number of bit comparisons required by {\tt Quickselect} to
find the smallest or largest key from a set of $n$ keys is
asymptotically linear in $n$ with the asymptotic slope approximately
equal to 5.27938. Hence asymptotically it differs from the expected
number of \emph{key} comparisons to achieve the same task only by a
constant factor. (The expectation for key comparisons is
asymptotically $2n$; see Knuth \cite{k1972} and Mahmoud \textit{et
al.} \cite{mms1995}). This result is rather contrastive to the
{\tt Quicksort} case in which the expected number of bit comparisons is
asymptotically $n(\ln n)(\lg n)$ whereas the expected number of key
comparisons is asymptotically $2n \ln n$ (see Fill and Janson
\cite{fj2004}).  Our analysis also showed that the expected number
of bit comparisons for the average case remains asymptotically
linear in $n$ with the lead-order coefficient approximately equal to
8.20731. Again, the expected number is asymptotically different from
that of key comparisons for the average case only by a constant
factor.  (The expected number of key comparisons for the average
case is asymptotically $3n$; see Mahmoud
\textit{et al.} \cite{mms1995}).\\
\indent Although we have yet to establish a formula analogous to
(\ref{eqnMu1}) and (\ref{eqnMuAvgFinal}) for the expected number of
bit comparisons to find the $m$-th key for fixed $m$, we established
an exact expression that only requires finite summation and used it
to obtain the results shown in Figure 1.  However, the formula
remains complex.  Written as a single expression, $\mu(m,n)$ is a
seven-fold sum of rather elementary terms with each sum having order
$n$ terms (in the worst case); in this sense, the running time of
the algorithm for computing $\mu(m,n)$ is of order $n^7$.  The
expression for $\mu(m,n)$ does not allow us to derive an asymptotic
formula for it or to prove the two intuitively obvious observations
described at the end of Section \ref{closedForm}. The situation is
substantially better for the expected number of \emph{key}
comparisons to find the $m$-th key from a set of $n$ keys; Knuth
\cite{k1972} showed that the expectation can be written
as $2[n+3+(n+1)H_n-(m+2)H_m-(n+3-m)H_{n+1-m}]$.\\
\indent In this paper, we considered independent and uniformly
distributed keys in (0,1). In this case, each bit in bit strings is
1 with probability 0.5. In our future research, we intend to
generalize the bit strings and consider each bit resulting from an
independent Bernoulli trial with parameter $p$.  This generalization
will further elucidate the bit complexity of {\tt Quickselect} and other
algorithms.\\

{\bf Acknowledgment.\ \ }We thank Philippe
Flajolet, Svante Janson, and Helmut Prodinger for helpful discussions.

\section{Appendix}\label{appendix}
\setlength\parindent{.5in}
In order to prove (\ref{complexIntegralAlongRVL}), it suffices to
show that, for any positive integer $m$,
\begin{eqnarray*}
\int_{n-\theta-i\infty}^{n-\theta+i\infty}\zeta(-1-s)\,m^{-s}\frac{ds}{(s+1)s\cdots
(s-n)} = 0
\end{eqnarray*}
(note that $n\geq 2$ and $0<\theta<1$). Letting $t:=-1-s$, it is
thus sufficient to show that
\begin{eqnarray*}
J:=\int_{-(n+1)+\theta-i\infty}^{-(n+1)+\theta+i\infty}\zeta(t)\,m^{t}\frac{dt}{t(t+1)\cdots
[t+(n+1)]} = 0.
\end{eqnarray*}
Using the residue theorem, we obtain
\begin{eqnarray}
J = - 2\pi i \left[\sum_{k=0}^{n}(-1)^k
\frac{\zeta(-k)\,m^{-k}}{k!(n+1-k)!}+\frac{m}{(n+2)!}\right] +
\int_{2-i\infty}^{2+i\infty}\zeta(t)\,m^t
\frac{dt}{t(t+1)\cdots[t+(n+1)]};\nonumber\\
\label{J}
\end{eqnarray}
The ``2'' in the second term here could just as well be any real
number exceeding 1. Here
\begin{eqnarray*}
\sum_{k=0}^{n}(-1)^k \frac{\zeta(-k)\,m^{-k}}{k!(n+1-k)!} =
-\frac{1}{2\,(n+1)!}+\sum_{k=1}^{n}
\frac{B_{k+1}\,m^{-k}}{(k+1)!(n+1-k)!}
=\sum_{k=1}^{n+1}
\frac{B_{k}\,m^{-(k-1)}}{k!(n+2-k)!}.
\end{eqnarray*}
Therefore
\begin{eqnarray}
\lefteqn{\sum_{k=0}^{n}(-1)^k
\frac{\zeta(-k)\,m^{-k}}{k!(n+1-k)!}+\frac{m}{(n+2)!}}\nonumber\\
&=& \frac{m^{-(n+1)}}{(n+1)!}\left[\sum_{k=1}^{n+1}
\frac{B_{k}\,(n+1)!}{k!(n+2-k)!}m^{n+2-k}+\frac{m^{n+2}}{n+2}\right]\nonumber\\
&=&\frac{m^{-(n+1)}}{(n+1)!}\sum_{k=1}^{m-1}k^{n+1}
=\frac{1}{(n+1)!}\sum_{k=1}^{m-1}\left(1-\frac{k}{m}\right)^{n+1};\label{partOfJ1}
\end{eqnarray}
for the second equality, see Knuth \cite{k2002v1} (Exercise
1.2.11.2-4). On the other hand, Flajolet \textit{et al.}
\cite{fgkpt1994} showed that
\begin{eqnarray}
& &\int_{2-i\infty}^{2+i\infty}\zeta(t)\,m^t
\frac{dt}{t(t+1)\cdots[t+(n+1)]} = \frac{2\pi
i}{(n+1)!}\sum_{k=1}^{m-1}\left(1-\frac{k}{m}\right)^{n+1}.\label{partOfJ2}
\end{eqnarray}
Thus it follows from (\ref{J})--(\ref{partOfJ2}) that $J=0$.

\bibliographystyle{plain}
\bibliography{references}
\end{document}